\documentclass[12pt]{article}
\usepackage{geometry,amsmath, amsfonts}
\usepackage{latexsym}
\usepackage{mathrsfs,color}

\usepackage{mhequ}
\usepackage{mhsymb}
\usepackage{mhenvs}
\usepackage{amssymb}

\newtheorem{assumption}[lemma]{Assumption}
\newtheorem{example}[lemma]{Example}

\numberwithin{equation}{section}

\def\w{w}

\def\loc{\mathrm{loc}}
\def\FF{\mathcal{F}}
\def\${|\!|\!|}
\def\/{\,\rule[-0.25em]{2pt}{1em}\,}
\def\Cov{\mathop{\mathbf{Cov}}}

\definecolor{darkred}{rgb}{0.9,0.1,0.1}

\newop{diag}

\title{Random homogenisation of a highly oscillatory singular potential}
\author{Martin Hairer, Etienne Pardoux, Andrey Piatnitski}
\date{}

\begin{document}
\maketitle

\begin{abstract}
In this article, we consider the problem of homogenising the linear heat equation
perturbed by a rapidly oscillating random potential. We consider the situation where
the space-time scaling of the potential's oscillations is \textit{not} given by
the diffusion scaling that leaves the heat equation invariant. Instead, we treat
the case where spatial oscillations are much faster than temporal oscillations.
Under suitable scaling of the amplitude of the potential, we prove convergence
to a deterministic heat equation with constant potential, thus completing
the results previously obtained in \cite{MR2962093}.
\end{abstract}

\section{Introduction}

We consider the parabolic PDE with space-time random potential given by
\begin{equs}\label{e:orig}
\partial_t u^\eps(x,t)&=\partial^2_x u^\eps(x,t)+\eps^{-\beta}
V\left(\frac{x}{\eps},\frac{t}{\eps^\alpha}\right)u^\eps(x,t)\;,\\
u^\eps(x,0)&=u_0(x)\;,
\end{equs}
where $x \in \R$, $t \ge 0$ and $V$ is a stationary centred random field.
The homogenisation theory of equations of this type has been studied by a number
of authors. The case when $V$ is time-independent was considered in \cite{MR2451056,MR2718269}. The articles \cite{MR1819484,MR2190162} considered a situation
where $V$ is a stationary process as a function of time, but periodic in space.
Purely periodic / quasiperiodic operators with large potential were also studied in
\cite{BLP,MR737902}.

For $\alpha \ge 2$ and $\beta = {\alpha \over 2}$, \eref{e:orig} was studied in \cite{MR2962093}, where it was shown that its solutions
converge as $\eps \to 0$ to the solutions to
\begin{equ}[e:limitu]
\partial_t u(x,t) = \partial^2_x u(x,t) + \bar V u(x,t)\;,\quad
u(x,0)=u_0(x)\;,
\end{equ}
where the constant $\bar V$ is given by
\begin{equ}[e:Vbarfast]
\bar V = \int_0^\infty \Phi(0,t) \,dt\;,
\end{equ}
in the case $\alpha > 2$ and
\begin{equ}[e:Vbardiff]
\bar V = \int_0^\infty \int_{-\infty}^\infty {e^{-{x^2 \over 4t}}\over 2\sqrt{\pi t}} \Phi(x,t) \,dx \,dt\;,
\end{equ}
in the case $\alpha = 2$. Here, $\Phi(x,t) = \E V(0,0)V(x,t)$ is the correlation function of $V$ which is assumed to
decay sufficiently fast.

In the case $0<\alpha<2$, it was conjectured in \cite{MR2962093} that the correct scaling to use in order to obtain
a non-trivial limit is $\beta = 1/2+\alpha/4$, but the corresponding value of $\bar V$ was not obtained.
Furthermore, the techniques used there seem to break down in this case.
The main result of the present article is that the conjecture does indeed hold true 
and that the solutions to \eref{e:orig}
do again converge to those of  \eref{e:limitu} as $\eps \to 0$. This time, the limiting constant $\bar V$ is given by
\begin{equ}[e:defVbar]
\bar V = \frac{1}{2\sqrt{\pi}}
\int_0^\infty \frac{\overline\Phi(t)}{\sqrt t}\,dt\;,
\end{equ}
where we have set $\overline\Phi(s):=\int_\R \Phi(x,s)dx$.

\begin{remark}
One can ``guess'' both \eref{e:Vbarfast} and \eref{e:defVbar} if we admit that \eref{e:Vbardiff} holds. Indeed, \eref{e:Vbarfast} is obtained from \eref{e:Vbardiff}
by replacing  $\Phi(x,t)$ by $\Phi(\delta x, t)$ and taking the limit $\delta \to 0$.
This reflects the fact that this corresponds to a situation in which, at the diffusive
scale, the temporal oscillations of the potential are faster than the
spatial oscillations. Similarly, \eref{e:defVbar} is obtained by replacing
$\Phi(x,t)$ with $\delta^{-1}\Phi(\delta^{-1} x, t)$ and then taking the limit
$\delta \to 0$, reflecting the fact that we are in the reverse situation where
spatial oscillations are faster. These arguments also allow to guess the
correct exponent $\beta$ in both regimes.
\end{remark}

The techniques employed in the present article are very different from \cite{MR2962093}: 
instead of relying on probabilistic techniques, we
adapt the analytical techniques from \cite{KPZ}.

From now on, we will rewrite \eref{e:orig} as
\begin{equ}
\partial_t u^\eps(x,t)=\partial^2_x u^\eps(x,t) + V_\eps(x,t)u^\eps(x,t)\;,\quad u^\eps(x,0) = u_0(x)\;,
\end{equ}
where $V_\eps$ is the rescaled potential given by
\begin{equ}
V_\eps(x,t) = \eps^{-(1/2+\alpha/4)}
V\left(\frac{x}{\eps},\frac{t}{\eps^\alpha}\right)\;.
\end{equ}
Before we proceed, we give a more precise description of our 
assumptions on the random potential $V$.

\subsection{Assumptions on the potential}

Besides some regularity and integrability assumptions, our main assumption will
be a sufficiently fast decay of maximal correlations for $V$.
Recall that the ``maximal correlation coefficient'' of $V$, subsequently denoted by $\rho$, is given
by the following definition where, for any given compact set $K\subset \R^2$, we denote by $\CF_K$
the $\sigma$-algebra generated by $\{V(x,t)\,:\, (x,t) \in K\}$.

\begin{definition}\label{def:cor}
For any $r > 0$, $\rho(r)$ is the smallest value such that the bound
\begin{equ}
\E \bigl(\phi_1(V)\phi_2(V)\bigr) \le \rho(r) \sqrt{\E \phi_1^2(V) \, \E \phi_2^2(V)}\;,
\end{equ}
holds for any two compact sets $K_1$, $K_2$ such that
\begin{equ}
d(K_1,K_2) \eqdef \inf_{(x_1,t_1) \in K_1}\inf_{(x_2,t_2) \in K_2} (|x_1 - x_2| + |t_1-t_2|) \ge r\;,
\end{equ}
and any two r.v.'s $\phi_i(V)$ such that $\phi_i(V)$ is $\FF_{K_i}$-measurable and $\E \phi_i(V) = 0$.
\end{definition}
Note that $\rho$ is a decreasing function.
With this notation at hand, we then make the following assumption:

\begin{assumption}\label{ass:Phi}
The field $V$ is stationary, centred, continuous, and $\CC^1$ in the $x$-variable.
Furthermore,
\begin{equ}
\E \bigl(|V(x,t)|^p + |\d_x V(x,t)|^p\bigr) < \infty
\end{equ}
for every $p > 0$.
\end{assumption}

For most of our results, we will furthermore require that the correlations of $V$ decay sufficiently fast in the following sense:

\begin{assumption}\label{ass:cor}
The maximal correlation function $\rho$ from Definition~\ref{def:cor} satisfies
$\rho(R) \lesssim (1+R)^{-q}$ for every $q > 0$.
\end{assumption}

\begin{remark}
Retracing the steps of our proof, one can see that in order to obtain our main result, Theorem~\ref{th:main},
we actually only need this bound for some sufficiently large $q$.
Similarly, the assumption on the $x$-differentiability of $V$ is not absolutely necessary, but simplifies
some of our arguments.
\end{remark}

Let us first give a few examples of random fields satisfying our assumptions.

\begin{example}\label{ex:Poisson}
Take a measure space $(\CM,\nu)$ with some finite measure $\nu$
and a function $\psi\colon \CM\times \R^2 \to \R$ such that
\begin{equ}
\sup_{m \in \CM} \sup_{x,t} {|\psi(m,x,t)| + |\d_x \psi(m,x,t)| \over  1 + |x|^q + |t|^q}< \infty\;,
\end{equ}
for all $q > 0$. Assume furthermore that $\psi$ satisfies the centering condition
\begin{equ}
 \int_\R \int_\R \int_{\CM} \psi(m,y,s)\,\nu(dm)\,dy\,ds = 0\;.
\end{equ} 
Consider now a realisation $\mu$ of the Poisson point process on $\CM\times \R^2$
with intensity measure $\nu(dm)\,dy\,ds$ and set
\begin{equ}
V(x,t) = \int_{\CM} \int_\R \int_\R \psi(m,y-x,s-t)\,\mu(dm,dy,ds)\;.
\end{equ}
Then $V$ satisfies Assumptions~\ref{ass:Phi} and \ref{ass:cor}.
\end{example}

\begin{example}\label{ex:Gaussian}
Take for $V$ a centred Gaussian field with covariance $\Phi$ such that
\begin{equ}
 \sup_{x,t} {|\Phi(x,t)| + |\d_x^2 \Phi(x,t)| \over  1 + |x|^q + |t|^q}< \infty\;,
\end{equ}
for all $q > 0$.
Then $V$ does not quite satisfy Assumptions~\ref{ass:Phi} and \ref{ass:cor} because $V$ and $\d_x V$ are
not necessarily continuous. However, it is easy to check that our proofs still work in this case.
\end{example}

The advantage of Definition~\ref{def:cor} is that it is invariant under the composition by
measurable functions.
In particular, given  a finite number of independent random fields $\{V_1,\ldots ,V_k\}$ of the type
of Examples~\ref{ex:Poisson} and \ref{ex:Gaussian} (or, more generally, any mutually independent fields  satisfying
Assumptions~\ref{ass:Phi} and \ref{ass:cor}) and a  function $F\colon \R^k \to \R$
such that
\begin{enumerate}
\item  $\E F(V_1(x,t),\ldots,V_k(x,t)) = 0$,
\item $F$, together with its first partial derivatives, grows no faster than polynomially at infinity.
\end{enumerate}
Then, our results hold with $V(x,t) = F(V_1(x,t),\ldots,V_k(x,t))$.

\subsection{Statement of the result}

Consider the solution to the heat equation with constant potential
\begin{equs}\label{e:limit}
\partial_t u(x,t)&=\partial^2_x u(x,t)+
\bar V u(x,t),\qquad t\ge0,x\in\R;\\
u(x,0)&=u_0(x),
\end{equs}
where $\bar V$ is defined by \eref{e:defVbar}. Then, the main result of this article
is the following convergence result:

\begin{theorem}\label{th:main}
Let $V$ be a random potential satisfying Assumptions~\ref{ass:Phi} and \ref{ass:cor}, and
let $u_0\in\CC^{3/2}(\R)$ be of no more than exponential growth.
Then, as $\eps\to0$, one has $u^\eps(t,x)\to u(t,x)$ in probability,
locally uniformly in $x \in\R$ and $t \ge 0$.
\end{theorem}

\begin{remark}
The precise assumption on $u_0$ is that it belongs to the space $\CC^{3/2}_{e_\ell}$ for
some $\ell \in \R$, see Section~\ref{sec:spaces} below for the definition of this space.
\end{remark}

\begin{remark}
The fact that $\E V = 0$ is of course not essential, since one can easily subtract the
mean by performing a suitable rescaling of the solution.
\end{remark}

To prove Theorem~\ref{th:main}, we use the standard ``trick'' to introduce a corrector
that ``kills'' the large potential $V_\eps$ to highest order. The less usual feature of
this problem is that, in order to obtain the required convergence, it turns out to be advantageous
to use \textit{two} correctors, which ensures that the remaining terms can be brought under control. These correctors, which we denote by $Y^\eps$ and $Z^\eps$,
are given by the solutions to the following inhomogeneous heat equations:
\begin{equs}[e:defCorrectors]
\partial_t Y^\eps(x,t)&=\partial^2_x Y^\eps(x,t)+V_\eps(x,t)\;,\\
\partial_t Z^\eps(x,t)&=\partial^2_x Z^\eps(x,t)+
\left|\partial_x Y^\eps(x,t)\right|^2-\bar V_\eps(t),
\end{equs}
where we have set $\bar V_\eps(t) = \E \left|\partial_x Y^\eps(x,t)\right|^2$.
In both cases, we start with the flat (zero) initial condition at $t=0$.
Writing
$$
v^\eps(x,t)=u^\eps(x,t)\exp\left[-\left(Y^\eps(x,t)+Z^\eps(x,t)\right)\right],
$$
Theorem~\ref{th:main} is then a consequence of the following two claims:
\begin{enumerate}
\item Both $Y^\eps$ and $Z^\eps$ converge locally
uniformly to $0$.
\item The process $v^\eps$ converges locally uniformly to the solution $u$
of \eref{e:limit}.
\end{enumerate}
It is straightforward to verify that $v^\eps$ solves the equation
\begin{equ}[e:veps]
\partial_t v^\eps=\partial^2_x v^\eps+\bar V_\eps\, v^\eps
+ 2
\left(\partial_x Y^\eps+\partial_x Z^\eps\right)\partial_x v^\eps+
\left(\left|\partial_x Z^\eps\right|^2
+2 \partial_x Z^\eps\partial_x Y^\eps \right)v^\eps\;,
\end{equ}
with initial condition $u_0$. The second claim will then  essentially follow from
the first (except that, due to the appearance of nonlinear terms involving the derivatives 
of the correctors, we need somewhat tighter control than just locally
uniform convergence), combined with the fact that
the function $\bar V_\eps(t)$ converges locally uniformly to the constant $\bar V$.

\begin{remark}
One way of ``guessing'' the correct forms for the correctors $Y^\eps$ and $Z^\eps$
is to note the analogy of the problem with that of building solutions to the KPZ
equation. Indeed, performing the Cole-Hopf transform
$h^\eps = \log u^\eps$, one obtains for
$h^\eps$ the equation
\begin{equ}
\d_t h^\eps = \d_x^2 h^\eps + \bigl(\d_x h^\eps\bigr)^2 + V_\eps\;,
\end{equ}
which, in the case where $V_\eps$ is replaced by space-time white noise, 
was recently analysed in detail in \cite{KPZ}.
The correctors $Y^\eps$ and $Z^\eps$ then arise naturally in this analysis
as the first terms in the Wild expansion of the KPZ equation.

This also suggests that it would be possible to find a diverging sequence
of constants $C_\eps$ such that the solutions to
\begin{equ}
\partial_t u^\eps(x,t)=\partial^2_x u^\eps(x,t)+\eps^{-{1+\alpha\over 2}}
V\left(\frac{x}{\eps},\frac{t}{\eps^\alpha}\right)u^\eps(x,t) - C_\eps u^\eps(x,t)\;,
\end{equ}
converge in law to the solutions to the multiplicative stochastic heat equation
driven by space-time white noise. In the non-Gaussian case, this does still seem out of
reach at the moment, although some recent progress can be found in \cite{HLong}.
\end{remark}

The proof of Theorem~\ref{th:main} now goes as follows. In a first step, which is
rather long and technical and constitutes Section~\ref{sec:bounds}
below, we obtain sharp a
priori bounds for $Y^\eps$ and $Z^\eps$ in various norms.
In a second step, which is performed in Section~\ref{sec:final}, we then combine these
estimates in order to show that the only terms in \eref{e:veps} that matter
are indeed the first two terms on the right hand side.

\begin{remark}
Throughout this article, the notation $X \lesssim Y$ will be equivalent
to the notation $X \le C Y$ for some constant $C$ independent of $\eps$.
\end{remark}

\section{Estimates of $Y^\eps$ and $Z^\eps$}
\label{sec:bounds}

In this section, we shall prove that both $Y^\eps$ and $Z^\eps$ 
tend to zero as $\eps\to0$, and establish further estimates
on those sequences of functions which will be needed for taking the limit of the sequence $v^\eps$. But before doing so, let us first introduce some technical tools which will be needed both in this section and in the last one.

\subsection{Weighted H\"older continuous spaces of functions and the heat semigroup}
\label{sec:spaces}

First of all, we define the notion of an admissible weight $\w$ as a function $\w \colon \R \to \R_+$ such that
there exists a constant $C\ge 1$ with
\begin{equ}[e:admissible]
C^{-1} \le {\w(x) \over \w(y)} \le C\;,
\end{equ}
for all pairs $(x,y)$ with $|x-y| \le 1$. Given such an admissible weight $\w$, we then define the
space $\CC_\w$ as the closure of $\CC_0^\infty$ under the norm
\begin{equ}
\|f\|_\w = \|f\|_{0,\w} = \sup_{x \in \R} { |f(x)| \over \w(x)}\;.
\end{equ}
We also define $\CC^\alpha_\w$ for $\alpha \in (0,1)$ as the closure of $\CC_0^\infty$ under the norm
\begin{equ}
\|f\|_{\alpha,\w} = \|f\|_\w + \sup_{|x-y| \le 1} {|f(x)-f(y)| \over \w(x) |x-y|^\alpha}\;.
\end{equ}
Similarly, for $\alpha \ge 1$, we define $\CC^\alpha_\w$ recursively as the closure of $\CC_0^\infty$ under the norm
\begin{equ}
\|f\|_{\alpha,\w} = \|f\|_\w + \|f'\|_{\alpha-1,\w}\;.
\end{equ}
It is clear that, if $\w_1$ and $\w_2$ are two admissible weights, then so is $\w = \w_1\,\w_2$. Furthermore,
it is a straightforward exercise to use the Leibniz rule to verify that there exists a constant $C$ such that the bound
\begin{equ}[e:boundMult]
\|f_1 f_2\|_{\alpha,\w} \le C \|f_1\|_{\alpha_1,\w_1} \|f_2\|_{\alpha_2,\w_2}\;,
\end{equ}
holds for every $f_i \in \CC_{\w_i}^{\alpha_i}$, provided that $\alpha \le \alpha_1 \wedge \alpha_2$.

We now show that a similar inequality still holds if one of the two H\"older exponents is \textit{negative}.
For $\alpha \in (-1,0)$, we can indeed define weighted spaces of negative ``H\"older regularity'' by
postulating that $\CC^{\alpha}_\w$ is the closure of $\CC_0^\infty$ under the norm
\begin{equ}
\|f\|_{\alpha,\w} = \sup_{|x-y| \le 1}  {|\int_x^y f(z)\,dz| \over \w(x) |x-y|^{\alpha+1}}\;.
\end{equ}
In other words, we essentially want the antiderivative of $f$ to belong to $\CC^{\alpha+1}_\w$, except
that we do not worry about its growth.

With these notations at hand, we then have the bound:

\begin{proposition}\label{prop:multHol}
Let $\w_1$ and $\w_2$ be two admissible weights and let $\alpha_1 < 0 < \alpha_2$ be such that
$\alpha_2 > |\alpha_1|$. Then, the bound \eref{e:boundMult} holds with $\alpha = \alpha_1$.
\end{proposition}

\begin{proof}
We only need to show the bound for smooth and compactly supported elements $f_1$ and $f_2$, the general case
then follows by density. Denote now by $F_1$ an antiderivative for $f_1$, so that
\begin{equ}
\int_x^y f_1(z)f_2(z)\,dz =\int_x^y f_2(z)\,dF_1(z)\;,
\end{equ}
where the right hand side is a Riemann-Stieltjes integral. For any interval $I\subset \R$, we now write
\begin{equ}
\/f\/_{\alpha,I} = \sup_{\{x,y\} \subset I} {|f(x) - f(y)| \over |x-y|^\alpha}\;.
\end{equ}
It then follows from Young's inequality \cite{Young} that there exists a constant $C$ depending only on the precise values of the $\alpha_i$
and on the constants appearing in the definition \eref{e:admissible} of admissibility for the weights $\w_i$, such that
\begin{equs}
\Bigl|\int_x^y f_2(z)\,dF_1(z)\Bigr| &\le |f_2(x)| \bigl|F_1(y)-F_1(x)\bigr| + C \/f_2\/_{\alpha_2,[x,y]} \/F_1\/_{\alpha_1+1,[x,y]} |x-y|^{\alpha_1 + \alpha_2 + 1} \\
&\le \w(x)|x-y|^{\alpha_1+1}\bigl( \|f_2\|_{0,\w_2} \|f_1\|_{\alpha_1,\w_1} + C \|f_2\|_{\alpha_2,\w_2} \|f_1\|_{\alpha_1,\w_1}\bigr)\;,
\end{equs}
which is precisely the requested bound.
\end{proof}

There are two types of admissible weights that will play a crucial role in the sequel:
\begin{equ}
e_\ell(x) \eqdef \exp(- \ell|x|)\;,\qquad p_\kappa(x) \eqdef 1 + |x|^\kappa\;,
\end{equ}
where the exponent $\kappa$ will always be positive, but $\ell$ could have any sign. One has of course the
identity
\begin{equ}[e:idexp]
e_\ell \cdot e_m = e_{\ell+m}\;.
\end{equ}
Furthermore, it is straightforward to verify that there exists a constant $C$ such that the bound
\begin{equ}[e:boundpe]
p_\kappa(x) e_\ell(x) \le C \ell^{-\kappa}\;,
\end{equ}
holds uniformly in $x \in \R$, $\kappa \in (0,1]$, and $\ell \in (0,1]$.

Finally, we have the following regularising property of the heat semigroup:

\begin{proposition}\label{prop:Heat}
Let $\alpha \in (-1,\infty)$, let $\beta > \alpha$, and let $\ell, \kappa \in \R$.
Then, for every $t>0$, the operator $P_t$ extends to a bounded
operator from $\CC^\alpha_{e_\ell}$ to $\CC^\beta_{e_\ell}$ and from
$\CC^\alpha_{p_\kappa}$ to $\CC^\beta_{p_\kappa}$.
Furthermore, for every $\ell_0 > 0$ and $\kappa_0 > 0$,
there exists a constant $C$ such that the bounds
\begin{equ}
\|P_t f\|_{\beta,e_\ell} \le C t^{-{\beta - \alpha\over 2}} \|f\|_{\alpha,e_\ell}\;,\qquad
\|P_t g\|_{\beta,p_\kappa} \le C t^{-{\beta - \alpha\over 2}} \|g\|_{\alpha,p_\kappa}\;,
\end{equ}
hold for every $f \in \CC_{e_\ell}^\alpha$, every $g \in \CC_{p_\kappa}^\alpha$, every $t \in (0,1]$, every
$|\ell| \le \ell_0$, and every $|\kappa| \le \kappa_0$.
\end{proposition}

\begin{proof}
The proof is standard: one first verifies that the semigroup preserves these norms,
so that the case $\beta = \alpha$ is covered. The case of integer values of $\beta$
can easily be verified by an explicit calculation. The remaining values then follow by
interpolation.
\end{proof}

\subsection{Bounds and convergence of $Y^\eps$ and $Z^\eps$}
For any integer $k\ge 2$, define the $k$-point correlation function $\Phi^{(k)}$
for $x, t \in \R^k$ by
\begin{equ}
\Phi^{(k)}(x,t) = \E \bigl(V(x_1,t_1)\ldots V(x_k, t_k)\bigr)\;.
\end{equ}
(In particular, $\Phi^{(2)}(x_1,t_1,x_2,t_2) = \Phi(x_1-x_2,t_1-t_2)$, where $\Phi$ is the
correlation function of $V$ defined above.) With these notations at hand, we
have the following bound which will prove to be useful:

\begin{lemma}\label{ass:cor4}
The function $\Psi^{(4)}$ given by
\begin{equ}
\Psi^{(4)}(x,t) = \Phi^{(4)}(x,t) - \Phi(x_1-x_2,t_1-t_2)\Phi(x_3-x_4,t_3-t_4)\;,
\end{equ}
satisfies the bound
\begin{equs}
|\Psi^{(4)}(x,t) |
&\le  \eta(|x_1-x_3|+|t_1-t_3|)\eta(|x_2-x_4|+|t_2-t_4|) \label{e:boundPsi}\\
&\qquad + \eta(|x_1-x_4|+|t_1-t_4|)\eta(|x_2-x_3|+|t_2-t_3|)\;,
\end{equs}
where the function $\eta\colon \R_+ \to \R_+$ is defined by
\begin{equation*}
\eta(r)=\sqrt{K\rho(r/3)},\qquad \text{\rm with }\
K=4\big(\|V(x,t)\|_{2}\|V^3(x,t)\|_{2}+ \|V^2(x,t)\|^2_{2}\big),
\end{equation*}
where we write $\|\cdot\|_2$ for the $L^2(\Omega)$ norm of a real-valued random variable.
\end{lemma}

\begin{remark}
In the Gaussian case, one has the identity
\begin{equ}
\Psi^{(4)}(x,t) = \Phi(x_1-x_3,t_1-t_3)\Phi(x_2-x_4,t_2-t_4) + \Phi(x_1-x_4,t_1-t_4)\Phi(x_2-x_3,t_2-t_3)\;,
\end{equ}
so that the bound \eref{e:boundPsi} follows from the 
fact that $\rho$ dominates the decay of the correlation function $\Phi$.
\end{remark}

\begin{proof}
For the sake of brevity denote $\xi_j=(x_j,t_j)$. We set
$$
R_1=\max\limits_{1\le i\le 4}\mathrm{dist}\Bigl(\xi_i,\bigcup\limits_{j\not=i}\{\xi_j\}\Bigr),\qquad
R_2=\max \mathrm{dist}\Bigl(\{\xi_{i_1},\xi_{i_2}\}, \{\xi_{i_3},\xi_{i_4}\}\Bigr),
$$
where the second maximum is taken over all permutations $\{i_1,i_2,i_3,i_4\}$ of $\{1,2,3,4\}$.

Consider first the case $R_1\geq R_2$. Without loss of generality we can assume that $R_1=\mathrm{dist}(\xi_1,\bigcup\limits_{j\not=1}\{\xi_j\})$.
 It is easily seen that, in the case under consideration,
\begin{equation}\label{geo_1}
\mathrm{dist}((\xi_i, \xi_j)\leq 3R_1,\qquad i,\,j=1,\,2,\,3,\,4.
\end{equation}
Then the functions $\Phi^{(4)}$ and $\Phi(\xi_1-\xi_2)\Phi(\xi_3-\xi_4)$ admit the following upper bounds:
\begin{equs}
|\Phi^{(4)}(\xi_1,\xi_2,\xi_3,\xi_4)|&=|\E(V(\xi_1)V(\xi_2)V(\xi_3)V(\xi_4))| \\
&\le \rho(R_1)\|V(\xi_1)\|_{2}\|V(\xi_2)V(\xi_3)V(\xi_4)\|_{2} \\
&\leq \rho(R_1)\|V(\xi)\|_{2}\|(V(\xi))^3\|_{2},
\end{equs}
and
$$
\Phi(\xi_1-\xi_2)\Phi(\xi_3-\xi_4)\le \rho(R_1)\|V\|^2_{2} \,\|V\|^2_{2}
$$
Therefore,
$$
|\Psi^{(4)}(x,t)|\leq\rho(\R_1)
\big(\ \|V(\xi)\|_{2}\|(V(\xi))^3\|_{2}+\|V\|^4_{2}\big)
$$
From (\ref{geo_1}) and the fact that $\rho$ is a decreasing function we derive
$$
K\rho(R_1)=\eta(3R_1)\eta(3R_1)\le \eta(|\xi_1-\xi_3|)\eta(|\xi_2-\xi_4|).
$$
This yields the desired inequality.

Assume now that $R_1<R_2$ and $\mathrm{dist}(\{\xi_1,\xi_2\}, \{\xi_3,\xi_4\})= R_2$. In this case
\begin{equation}\label{geo_2}
\mathrm{dist}(\xi_1,\xi_2)<R_2 \quad\hbox{and }\ \ \mathrm{dist}(\xi_3,\xi_4)<R_2.
\end{equation}
Indeed, if we assume that $\mathrm{dist}(\xi_1,\xi_2)\geq R_2$, then $\mathrm{dist}(\xi_1,\{\xi_2,\xi_3,\xi_4\})\geq R_2$ and, thus, $R_1\geq R_2$ which contradicts our assumption.
We have
\begin{equs}\label{geo_3}
\big|\Psi^{(4)}&(\xi_1,\xi_2,\xi_3,\xi_4)\big|=
\big|\Phi^{(4)}(\xi_1,\xi_2,\xi_3,\xi_4)-\Phi(\xi_1-\xi_2)\Phi(\xi_3-\xi_4)\big| \\
&=\big|\E\big([V(\xi_1)V(\xi_2)-\E(V(\xi_1)V(\xi_2))] [V(\xi_3)V(\xi_4)-\E(V(\xi_3)V(\xi_4))]\big)\big|\\
&\leq \rho(R_2)\|(V(\xi))^2\|^2_{2}\;.
\end{equs}
In view of (\ref{geo_2}), $\mathrm{dist}(\xi_1,\xi_3)\le 3R_2$ and $\mathrm{dist}(\xi_2,\xi_4)\le 3R_2$. Therefore,
$$
K\rho(R_2)\le\eta(|\xi_1-\xi_3|)\eta(|\xi_2-\xi_4|),
$$
and the desired inequality follows.

It remains to consider the case $R_1<R_2$ and $\mathrm{dist}(\{\xi_1,\xi_3\}, \{\xi_2,\xi_4\})= R_2$; the case $\mathrm{dist}(\{\xi_1,\xi_4\}, \{\xi_2,\xi_3\})= R_2$
can be addressed in the same way. In this case
$$
\mathrm{dist}(\xi_1,\xi_2)\geq R_2,\quad \mathrm{dist}(\xi_1,\xi_4)\geq R_2, \quad
\mathrm{dist}(\xi_1,\xi_3)< R_2.
$$
Therefore, $\mathrm{dist}(\xi_1,\{\xi_2,\xi_3,\xi_4\})= \mathrm{dist}(\xi_1,\xi_3)$, and we have
\begin{equs}
|\Phi^{(4)}(\xi_1,\xi_2,\xi_3,\xi_4)|&\le \rho(|\xi_1-\xi_3|)
\|V(\xi)\|_{2}\|(V(\xi))^3\|_{2}\\
|\Phi(\xi_1-\xi_2)\Phi(\xi_3-\xi_4)|&\le \rho(R_2)\|V\|^4_{2}\le  \rho(|\xi_1-\xi_3|) \|V\|^4_{2}.
\end{equs}
This yields
$$
|\Psi^{(4)}(\xi_1,\xi_2,\xi_3,\xi_4)|\le \rho(|\xi_1-\xi_3|)
\big(\|V(\xi)\|_{2}\|(V(\xi))^3\|_{2}+
\|V\|^4_{2}\big)
$$
In the same way one gets
$$
|\Psi^{(4)}(\xi_1,\xi_2,\xi_3,\xi_4)|\le \rho(|\xi_2-\xi_4|)
\big(\|V(\xi)\|_{2}\|(V(\xi))^3\|_{2}+
\|V\|^4_{2}\big)
$$
From the last two estimates we obtain
\begin{equs}
|\Psi^{(4)}(\xi_1,\xi_2,\xi_3,\xi_4)|&\leq\sqrt{\rho(|\xi_1-\xi_3|)}
\sqrt{\rho(|\xi_2-\xi_4|)} \big(\|V(\xi)\|_{2}\|(V(\xi))^3\|_{2}+
\|V\|^4_{2}\big)\\
&\leq \eta(|\xi_1-\xi_3|)\eta(|\xi_2-\xi_4|).
\end{equs}
This implies the desired inequality and completes the proof of Lemma \ref{ass:cor4}.
\end{proof}

In order to prove our next result, we will need the following small lemma:
\begin{lemma}\label{lem:intPower}
Let $F\colon \R_+ \to \R_+$ be an increasing function with $F(r) \le r^q$.
Then, $\int_0^\infty (1+r)^{-p} dF(r) < \infty$ as soon as $p > q$.
\end{lemma}
\begin{proof}
We have $\int_0^\infty (1+r)^{-p} dF(r) \le 1 + \int_1^\infty r^{-p} dF(r)$, so we only 
need to bound the latter. We write
\begin{equ}
\int_1^\infty r^{-p} dF(r) \le \sum_{k \ge 0} \int_{2^k}^{2^{k+1}}r^{-p}dF(r)
\le \sum_{k \ge 0}2^{-pk} \int_{2^k}^{2^{k+1}}dF(r) \le \sum_{k \ge 0}2^{-pk} 2^{q(k+1)}\;.
\end{equ}
This expression is summable as soon as $p > q$, thus yielding the claim.
\end{proof}
	
\begin{lemma}\label{lem:moment_int_c}
Fix $t > 0$ and let $\varphi :\R\times\R_+\to\R_+$ be a smooth function with compact support.
Define $\varphi_\delta(x,t)=\delta^{-3}\varphi\left(\frac{x}{\delta},\frac{t}{\delta^2}\right)$. Then, for all $p\ge1$, $\eps, \delta>0$, one has the bound
\begin{equs}
{}&\left[\E\left(\int_0^t\int_\R \varphi_\delta(x-y,t-s)V_\eps(y,s)\,dy\,ds\right)^p\right]^{1/p}\\
&\qquad\qquad\le C_\phi \bigl(\eps^{-1/2-\alpha/4}\wedge \delta^{-1/2}\eps^{-\alpha/4}\wedge \delta^{-3/2}\eps^{\alpha/4} \bigr)\;,
\end{equs}
where $C_\phi$ depends on $p$, on the supremum and the support of $\phi$, and on the bound of Assumption~\ref{ass:Phi}.
\end{lemma}

\begin{proof}
We consider separately the cases $\delta>\max(\eps,\eps^{\alpha})$, $\delta<\min(\eps,\eps^\alpha)$, 
as well as  $\min(\eps,\eps^\alpha)\leq \delta\leq \max(\eps,\eps^{\alpha})$.

Assume first that $\delta>\max(\eps,\eps^{\alpha})$. Without loss of generality we also assume that $p$ is even, that is $p=2k$ with $k\in\mathbb N$.
Then
\begin{equs}
\mathcal{J}_p^{\eps,\delta}&:=\E\left(\int_0^t\int_\R \varphi_\delta(x-y,t-s)V_\eps(y,s)\,dy\,ds\right)^p\\
&\;=\int_0^t\!\!\!\!\dots\!\!\int_0^t\int_\R\!\!\!\!\dots\!\!\int_\R \prod\limits_{i=1}^{2k}\varphi_\delta(x-y_i, t-s_i)
\E\Big(\prod\limits_{i=1}^{2k}V_\eps(y_i,s_i)\Big)d\vec yd\vec s,
\end{equs}
where $d\vec y=dy_1\dots dy_{2k}$ and $d\vec s=ds_1\dots ds_{2k}$. Changing the variables $\tilde y_i=\eps^{-1}y_i$ and
$\tilde s_i=\eps^{-\alpha}s_i$, and considering the definition of $\varphi_\delta$ and $V_\eps$, we obtain
\begin{equ}
\mathcal{J}_p^{\eps,\delta}=\delta^{-6k}\eps^{-k-\frac{\alpha k}{2}}\eps^{2k+2\alpha k}
\int\limits_{[0,t/\eps^\alpha]^{2k}}\int\limits_{\mathbb R^{2k}}\prod\limits_{i=1}^{2k}\varphi\Big(\frac{x-\eps\tilde y_i}{\delta},\frac{t-\eps^\alpha\tilde s_i}{\delta^2}\Big)
\E\Big(\prod\limits_{i=1}^{2k}V(\tilde y_i,\tilde s_i)\Big)
d\vec{\tilde y}d\vec{\tilde s}.
\end{equ}
The support of the function $\prod\limits_{i=1}^{2k}\varphi\big(\frac{x-\eps\tilde y_i}{\delta},\frac{t-\eps^\alpha\tilde s_i}{\delta^2}\big)$
belongs to the rectangle
$(x-k\frac{\delta}{\eps}s_\varphi,x+k\frac{\delta}{\eps}s_\varphi)^{2k}
\times(t-k\frac{\delta^2}{\eps^\alpha}s_\varphi,t+ k\frac{\delta^2}{\eps^\alpha}s_\varphi)^{2k}$, where
$s_\varphi$ is the diameter of support of $\varphi=\varphi(y,s)$.
Denote
$\Pi^1_{\delta,\eps}=(0,2k\frac{\delta}{\eps}s_\varphi)^{2k}$
and $\Pi^2_{\delta,\eps}=(0,2k\frac{\delta^2}{\eps^\alpha}s_\varphi)^{2k}$.
Since $V(y,s)$ is stationary, we have
\begin{equation}\label{calj}
\mathcal{J}_p^{\eps,\delta}\le
\delta^{-6k}\eps^{-k-\frac{\alpha k}{2}}\eps^{2k+2\alpha k}
\|\varphi\|\big._{C}^{2k}\int\limits_{(0,2k\frac{\delta}{\eps}s_\varphi)^{2k}}
\int\limits_{(0,2k\frac{\delta^2}{\eps^\alpha}s_\varphi)^{2k}}
\Big|\E\Big(\prod\limits_{i=1}^{2k}V(\tilde y_i,\tilde s_i)\Big)
\Big|d\vec{\tilde y}d\vec{\tilde s}.
\end{equation}
For any $R\ge 0$ we introduce a subset of $\mathbb R^{4k}$
$$
\mathcal{V}_{\delta,\eps}(R)=\big\{(\tilde y,\tilde s)\in \Pi^1_{\delta,\eps}\times\Pi^2_{\delta,\eps}\,:\,
\max\limits_{1\le j\le 2k}\mathrm{dist}(\tilde y_j,\bigcup\limits_{i\not=j}\tilde y_i)\le R,\ \max\limits_{1\le j\le 2k}\mathrm{dist}(\tilde s_j,\bigcup\limits_{i\not=j}\tilde s_i)\le R\big\},
$$
and denote by $|\mathcal{V}_{\delta,\eps}|(R)$ the Lebesgue measure of this set.
It is easy to check that the set $\mathcal{V}_{\delta,\eps}(0)$
is the union of sets of the form
$$
\{(\tilde y,\tilde s)\in \Pi^1_{\delta,\eps}\times\Pi^1_{\delta,\eps}\,:\, \tilde y_{i_1}=\tilde y_{i_2}, \dots, \tilde y_{i_{2k-1}}=\tilde y_{i_{2k}},\  \tilde s_{j_1}=\tilde s_{j_2}, \dots, \tilde s_{j_{2k-1}}=\tilde s_{j_{2k}}\}
$$
with $i_l\not=i_m$ and $j_l\not=j_m$ if $l\not=m$, that is,
$\mathcal{V}_{\delta,\eps}(0)$ is the union of a finite number of subsets of $2k$-dimensional planes in $R^{4k}$.
The $2k$-dimensional measure of this set satisfies the
following upper bound
$$
|\mathcal{V}_{\delta,\eps}(0)|_{2k}\le C(k)\Big(\frac{\delta}{\eps}\Big)^{k} \Big(\frac{\delta^2}{\eps^\alpha}\Big)^{k},
$$
Therefore,
\begin{equation}\label{vr_bou}
|\mathcal{V}_{\delta,\eps}|(R)\lesssim \Big(\frac{\delta}{\eps}\Big)^{k} \Big(\frac{\delta^2}{\eps^\alpha}\Big)^{k} R^{2k},
\end{equation}
For each $(\tilde y, \tilde s)\in \mathcal{V}_{\delta,\eps}(R)$ we have
\begin{equation}\label{cc_bou}
\Big|\E\Big(\prod\limits_{i=1}^{2k}V(\tilde y_i,\tilde s_i)\Big)\Big|
\le \rho(R)C_1(k)\|V\|_{L^2(\Omega)}\|V^{2k-1}\|_{L^2(\Omega)}.
\end{equation}
Combining (\ref{calj}), (\ref{vr_bou}) and (\ref{cc_bou})
yields
\begin{equ}
\mathcal{J}_p^{\eps,\delta}\lesssim \delta^{-6k}\eps^{-k-\frac{\alpha k}{2}}\eps^{2k+2\alpha k}
  \int_0^\infty \rho(R)\,d|\mathcal{V}_{\delta,\eps}|(R) \lesssim  \delta^{-3k}\eps^{\frac{\alpha k}{2}}\;.
\end{equ}
Here, the last inequality holds due to Assumption~\ref{ass:cor}, combined with 
\eref{vr_bou} and Lemma~\ref{lem:intPower}. Therefore, recalling that $p=2k$, we have the bound
\begin{equation}\label{j_1}
(\mathcal{J}_p^{\eps,\delta})^{1/p}\lesssim \delta^{-3/2}\eps^{\alpha/4}.
\end{equation}

In the case $\delta<\min(\eps,\eps^\alpha)$ we have
\begin{equs}
\mathcal{J}_p^{\eps,\delta}
&=\int_0^t\!\!\!\!\dots\!\!\int_0^t\int_\R\!\!\!\!\dots\!\!\int_\R \prod\limits_{i=1}^{2k}\varphi_\delta(x-y_i, t-s_i)
\E\Big(\prod\limits_{i=1}^{2k}V_\eps(y_i,s_i)\Big)
d\vec y\,d\vec s\\
&\le\int_0^t\!\!\!\!\dots\!\!\int_0^t\int_\R\!\!\!\!\dots\!\!\int_\R \prod\limits_{i=1}^{2k}|\varphi_\delta(x-y_i, t-s_i)|\,
\Bigl|\E\Big(\prod\limits_{i=1}^{2k}V_\eps(y_i,s_i)\Big)\Bigr|
d\vec y\,d\vec s\\
&\le \E\big((V_\eps(y_1,s_1)^{2k}\big)\int_0^t\!\!\!\!\dots\!\!\int_0^t \int_\R\!\!\!\!\dots\!\!\int_\R \prod\limits_{i=1}^{2k}|\varphi_\delta(x-y_i, t-s_i)|d\vec y\,d\vec s\\
&\lesssim \eps^{-k-\frac{\alpha k}{2}}\|\varphi\|_{L^1}^{2k}\;,
\end{equs}
so that
\begin{equation}\label{j_2}
(\mathcal{J}_p^{\eps,\delta})^{1/p}\lesssim \eps^{-1/2-\alpha/4}\;.
\end{equation}

Finally, if we are in the regime $\eps<\delta<\eps^{\alpha/2}$, then
\begin{equs}
\mathcal{J}_p^{\eps,\delta}&=\delta^{-6k}\eps^{k+{3\alpha\over 2} k}
 \int\limits_{[0,t/\eps^\alpha]^{2k}}\int\limits_{\mathbb R^{2k}}\prod\limits_{i=1}^{2k}\varphi\Big(\frac{x-\eps\tilde y_i}{\delta},\frac{t-\eps^\alpha\tilde s_i}{\delta^2}\Big)
\E\Big(\prod\limits_{i=1}^{2k}V(\tilde y_i,\tilde s_i)\Big)
d\vec{\tilde y}\,d\vec{\tilde s}\\
&\le \delta^{-6k}\eps^{k+3\alpha k/2}\,
\|\varphi\|_{L^\infty}^{2k}
\int\limits_{(0,2k\frac{\delta^2}{\eps^\alpha}s_\varphi)^{2k}}
\int\limits_{(0,2k\frac{\delta}{\eps}s_\varphi)^{2k}}
\Big|\E\Big(\prod\limits_{i=1}^{2k}V(\tilde y_i,\tilde s_i)\Big)\Big|
d\vec{\tilde y}\,d\vec{\tilde s}\\
&\lesssim\delta^{-6k}\eps^{k+3\alpha k/2}\,
\|\varphi\|_{L^\infty}^{2k}
\int\limits_{(0,2k\frac{\delta^2}{\eps^\alpha}s_\varphi)^{2k}}\!\!\!\!\!
\|V\|_{L^2(\Omega)}\|V^{2k-1}\|_{L^2(\Omega)}
\Big(\frac{\delta}{\eps}\Big)^k\int\limits_0^\infty\rho(R)R^{k-1}\,dR\\
&\lesssim \delta^{-k}\eps^{-\alpha k/2}.
\end{equs}
Hence,
\begin{equation}\label{j_3}
(\mathcal{J}_p^{\eps,\delta})^{1/p}\lesssim \delta^{-1/2}\eps^{-\alpha/4}
\end{equation}
so that, combining (\ref{j_1}), (\ref{j_2}) and (\ref{j_3}), the desired estimate holds.
\end{proof}


\begin{lemma}\label{lem:moment_int_c0}
Fix $t > 0$ and let $\varphi :\R\times\R_+\to\R_+$ be a function which is uniformly bounded and
decays exponentially in $x$, uniformly over $s \in [0,t]$.

Then, for all $p\ge1$, $\eps>0$, one has the bound
\begin{equ}
\left[\E\left(\int_0^t\int_\R \varphi(x-y,t-s)V_\eps(y,s)\,dy\,ds\right)^p\right]^{1/p}
\le C_\phi \left(\eps^{-1/2-\alpha/4}\wedge \eps^{-\alpha/4}\wedge \eps^{\alpha/4}
\right)\;.
\end{equ}
Here, the proportionality constant depends on $p$, on $t$, on the bounds on $\phi$, and on the bounds
of Assumption~\ref{ass:Phi}.
\end{lemma}

\begin{proof}
The proof of this lemma is similar (with some simplifications) to that of the previous statement.
We leave it to the reader.
\end{proof}

\begin{lemma}\label{le:momentsY}
For each $p\ge1$, there exists a constant $C_p$ such that for all $\eps>0$, $t\ge0$, $x\in\R$,
\begin{align}\label{e:momentsY}
\left[\E\left(\left|Y^\eps(x,t)\right|^p\right)\right]^{1/p}&\le C_p(1+\sqrt t)\eps^{\alpha/4}\\
\label{e:momentsYderiv}
\left[\E\left(\left|\partial_x Y^\eps(x,t)\right|^p\right)\right]^{1/p}&\le C_p\\
\label{e:momentsYderiv2}
\left[\E\left(\left|\partial^2_x Y^\eps(x,t)\right|^p\right)\right]^{1/p}&\le C_p\eps^{-1}.
\end{align}
\end{lemma}

\begin{proof}
Our main ingredient is the existence of a function $\psi\colon \R_+ \to [0,1]$ which is
smooth, compactly supported in the interval $[1/2,2]$, and such that
\begin{equ}
\sum_{n \in \Z} \psi(2^{-n}r) = 1\;,
\end{equ}
for all $r > 0$.

As a consequence, we can rewrite the heat kernel as
\begin{equ}[e:rewritePt]
p_t(x) = \sum_{n \in \Z} 2^{-2n}\phi_n(x,t)\;,
\end{equ}
where
\begin{equ}[e:defphin]
\phi_n(x,t) = 2^{3n} \phi(2^{n}x, 2^{2n} t)\;,\qquad
\phi(x,t) = p_t(x) \psi(\sqrt{x^2 + t})\;.
\end{equ}
The advantage of this formulation is that the function $\phi$ is smooth and
compactly supported. The reason why we scale $\phi_n$ in this way, at the expense of still
having a prefactor $2^{-2n}$ in \eref{e:rewritePt} is that this is the scaling used in
Lemma~\ref{lem:moment_int_c} (setting $\delta = 2^{-n}$).

We use this decomposition to define $Y^\eps_n$ by
\begin{equ}[e:decompY]
Y^\eps_n(x,t) = 2^{-2n} \int_0^t \int_\R \phi_n(x-y, t-s)\, V_\eps(y,s)\,dy\,ds\;,
\end{equ}
so that, by \eref{e:rewritePt}, one has $Y^\eps = \sum_n Y^\eps_n$.
Setting $\tilde \phi(x,t) = \d_x \phi(x,t)$ and defining $\tilde \phi_n(x,t)= 2^{3n} \tilde \phi(2^{n}x, 2^{2n} t)$ as in \eref{e:defphin},
the derivative of $Y^\eps$ can be decomposed in the same way:
\begin{equ}[e:decompDY]
\d_x Y^\eps_n(x,t) = 2^{-n} \int_0^t \int_\R \tilde \phi_n(x-y, t-s)\, V_\eps(y,s)\,dy\,ds\;.
\end{equ}

We first bound the derivative of $Y^\eps$. Since $\tilde \phi$ is smooth and compactly supported, the constants appearing in  Lemma~\ref{lem:moment_int_c} do not depend on $t$ and we have
\begin{equ}
\bigl(\E |\d_x Y^\eps_n(x,t)|^p\bigr)^{1/p} \lesssim 2^{n/2} \eps^{\alpha/4} \wedge 2^{-n/2} \eps^{-\alpha/4}= 2^{-|{n\over 2} + {\alpha \over 4}\log_2 \eps|}
\;.
\end{equ}
Since the sum (over $n$) of this quantity is bounded independently of $\eps$,
\eqref{e:momentsYderiv} now follows by the triangle inequality.

Note that \eqref{e:momentsYderiv2} follows from the same argument, if we integrate by parts (hence differentiate $V_\eps$).

In order to finally establish \eqref{e:momentsY}, we bound $Y^\eps$ in a similar way. This time however, we combine all the terms with $n < 0$ into one single term,
setting
\begin{equ}
p_t^{-}(x) = \sum_{n \le 0} 2^{-2n} \phi_n(x,t)\;,\qquad Y^\eps_-(x,t) = \int_0^t \int_\R p_{t-s}^-(x-y)\, V_\eps(y,s)\,dy\,ds\;,
\end{equ}
so that $Y^\eps = \sum_{n > 0} Y^\eps_n + Y^\eps_-$. Similarly to before, we obtain
\begin{equ}[e:boundYn]
\bigl(\E |Y^\eps_n(x,t)|^p\bigr)^{1/p} \lesssim 2^{-n/2} \eps^{\alpha/4}\;.
\end{equ}
In order to bound $Y^\eps_-$, we apply Lemma~\ref{lem:moment_int_c} with $\delta = 1$ and $\phi = p^-$.
It is immediate that $c_1(t) \lesssim \sqrt t$ and $c_3(t)\lesssim 1$, so that
\begin{equ}
\bigl(\E |Y^\eps_-(x,t)|^p\bigr)^{1/p} \lesssim \sqrt t\eps^{\alpha/4}\;.
\end{equ}
Combining this with \eref{e:boundYn}, summed over $n > 0$, yields the desired bound.
\end{proof}

We deduce from Lemma \ref{le:momentsY} and equation \ref{e:defCorrectors}
\begin{corollary}\label{co:Yto0} 
As $\eps\to0$, $Y^\eps(x,t)\to0$ in probability, locally uniformly with respect
to $x$ and $t$.
\end{corollary}

\begin{proof}
It follows from  Lemma \ref{le:momentsY} and equation \ref{e:defCorrectors} that for some $a,b>0$ and all $p\ge1$,
all  bounded subsets $D\subset\R_+\times\R$,
\begin{align}\label{e:estimY}
\sup_{(x,t)\in D}&\E[|Y^\eps(x,t)|^p]\lesssim\eps^{pa},\\  \sup_{(x,t)\in D}\E[|\partial_xY^\eps(x,t)|^p]&\lesssim\eps^{-pb},\
 \sup_{(x,t)\in D}\E[|\partial_tY^\eps(x,t)|^p]\lesssim\eps^{-pb}. \label{e:estimdY}
 \end{align}
 We deduce from \eqref{e:estimY} that for all $(x,t), (y,s)\in D$, $p\ge1$,
 $$\E[|Y^\eps(x,t)-Y^\eps(y,s)|^p]\lesssim\eps^{pa},$$
 and from \eqref{e:estimdY}, writing $Y^\eps(x,t)-Y^\eps(y,s)$ as the sum of an integral of $\partial_xY^\eps$ and an integral of $\partial_tY^\eps$, we get
  $$\E[|Y^\eps(x,t)-Y^\eps(y,s)|^p]\lesssim (|x-y|+|t-s|)^p\eps^{-pb}.$$
  Hence from H\"older's inequality
  $$\E[|Y^\eps(x,t)-Y^\eps(y,s)|^{\alpha+\beta}]\le(|x-y|+|t-s|)^\beta \eps^{\alpha a-\beta b}.$$
  Provided $\beta>2$ and $\alpha>\beta b/a$, we obtain an estimate which allows us to deduce the result from a combination of \eqref{e:estimY} and Kolmogorov's Lemma.
\end{proof}

We will also need
\begin{lemma}\label{le:bar(c)}
The function $t\to\bar V_\eps(t)$ is continuous, and, for each $\eps>0$,
there exists a positive constant $\bar V^0_\eps$
such that
$$\bar V_\eps(t)\to\bar V_\eps^0,\quad\text{as }t\to\infty.
$$
Furthermore,
$$
\lim\limits_{\eps\to 0}\bar V_\eps^0=\bar V:=
\int\limits_0^\infty \int\limits_{\mathbb R} \frac{\Phi(y,t)}{2\sqrt{\pi t}}\,dy\,dt,
$$
and
$\bar V_\eps(t)\to\bar V$ as $\eps\to 0$,
uniformly in $t\in[1,+\infty]$.
\end{lemma}

\begin{proof}
Writing $\Phi_\eps$ for the correlation function of $V_\eps$
and using the definition of $\bar V_\eps(t)$, we have
\begin{equs}
\bar V_\eps(t) &=
\E\bigg[\bigg(\frac{\partial}{\partial x}\int\limits_0^t
\int\limits_{\mathbb R}p_{t-s}(x-y) V_\eps(y,s)\,dy\,ds\bigg)^2 \bigg]\\
&=\E\bigg[\bigg(\int\limits_0^t
\int\limits_{\mathbb R}p_{t-s}'(x-y) V_\eps(y,s)\,dy\,ds\bigg)^2 \bigg]\\
&=\E\bigg[\int\limits_0^t\!\!\!\int\limits_0^t\!\!\!
\int\limits_{\mathbb R}\!\!\!\int\limits_{\mathbb R}
p_{t-s}'(x-y)p_{t-r}'(x-z)
V_\eps(y,s) V_\eps(z,r)\,dy\,dz\,ds\,dr
\bigg]\\
&=\int\limits_0^t\!\!\!\int\limits_0^t\!\!\!
\int\limits_{\mathbb R}\!\!\!\int\limits_{\mathbb R}
p_{t-s}'(x-y)p_{t-r}'(x-z)
\Phi_\eps(y-z,s-r)\,dy\,dz\,ds\,dr\\
&=\int\limits_0^t\!\!\!\int\limits_0^t\!\!\!
\int\limits_{\mathbb R}\!\!\!\int\limits_{\mathbb R}
p_{s}'(y)p_{r}'(z)
\Phi_\eps(y-z,s-r)\,dy\,dz\,ds\,dr\;.
\end{equs}
It is easy to check that, for each $\eps>0$, this integral is a continuous function of $t$ and that it converges, as $t\to+\infty$.
Performing the change of variables $y'=\frac{y}{\eps^{1/2+\alpha/4}},\, z'=\frac{z}{\eps^{1/2+\alpha/4}},\,s'=\frac{s}{\eps^{1+\alpha/2}},\,
r'=\frac{r}{\eps^{1+\alpha/2}}$, renaming the new variables and setting $T_\eps=\eps^{-1-\alpha/2}t$, we obtain
$$
\bar V_\eps(t) =\frac{1}{16\pi} \int\limits_0^{T_\eps}\!\! \int\limits_0^{T_\eps}\!\!
\int\limits_{\mathbb R}\!\!\!\int\limits_{\mathbb R} \frac{y}{s^{3/2}} \frac{z}{r^{3/2}} e\big.^{-\frac{y^2}{4s}-\frac{z^2}{4r}}
\Phi\Big(\frac{y\!-\!z}{\eps^{\frac{1}{2}-\frac{\alpha}{4}}}, \frac{s\!-\!r}{\eps^{\frac{\alpha}{2}-1}}\Big)
\,dy\,dz\,ds\,dr.
$$
We represent the integral on the right-hand side as
\begin{equation}\label{repr_main}
\bar V_\eps(t)=
\frac{1}{16\pi} \int\limits_0^{T_\eps}\!\! \int\limits_0^{T_\eps}\!\!
\int\limits_{\mathbb R}\!\!\!\int\limits_{\mathbb R} \frac{z^2}{s^{3/2} r^{3/2}} e\big.^{-\frac{z^2}{4s}-\frac{z^2}{4r}}
\Phi\Big(\frac{y\!-\!z}{\eps^{\frac{1}{2}-\frac{\alpha}{4}}}, \frac{s\!-\!r}{\eps^{\frac{\alpha}{2}-1}}\Big)
\,dy\,dz\,ds\,dr + r_\eps(t).
\end{equation}
The further analysis relies on the following limit relation:
\begin{equation}\label{est_cruci}
\lim\limits_{\eps\to0}\ \sup_{0<t\le+\infty}\,|r_\eps(t)|=0.
\end{equation}
In order to justify it we denote $\varkappa=\frac{1}{2}-\frac{\alpha}{4}$ and $\varkappa_1=\frac{\varkappa}{10}$, and divide the integration area into four parts as follows
\begin{equs}
\Pi_1&=\{(y,z,s,r)\in\mathbb R^2\times(\mathbb R^+)^2\,:\,s\le \eps^{\varkappa_1},\,r\le \eps^{\varkappa_1}\},\\
\Pi_2&=\{(y,z,s,r)\in\mathbb R^2\times(\mathbb R^+)^2\,:\, \eps^{\varkappa_1}<s\le T_\eps ,\,r\le \eps^{\varkappa_1}\},\\
\Pi_3&=\{(y,z,s,r)\in\mathbb R^2\times(\mathbb R^+)^2\,:\,s\le \eps^{\varkappa_1},\, \eps^{\varkappa_1}<r\le T_\eps\},\\
\Pi_4&=\{(y,z,s,r)\in\mathbb R^2\times(\mathbb R^+)^2\,:\,\eps^{\varkappa_1}<s\le T_\eps,\,\eps^{\varkappa_1}<r\le T_\eps\}.
\end{equs}
In $\Pi_1$ we have
\begin{equation}\label{pi-1}
\int\limits_{\Pi_1}\frac{|y|\,|z|}{s^\frac{3}{2}r^\frac{3}{2}}
e\big.^{-\frac{y^2}{4s}-\frac{z^2}{4r}}
\Phi\Big(\frac{y\!-\!z}{\eps^{\frac{1}{2}-\frac{\alpha}{4}}}, \frac{s\!-\!r}{\eps^{\frac{\alpha}{2}-1}}\Big)\,dx\,dy\,ds\,dr\le C^2\int\limits_0^{\eps^{\varkappa_1}}\int\limits_0^{\eps^{\varkappa_1}}
\,\frac{dsdr}{s^\frac{1}{2}r^\frac{1}{2}}=4C^2\eps^{\varkappa_1}.
\end{equation}
To estimate the integral over $\Pi_2$ we first notice that there exists a constant
$C_1$ such that
$$
\frac{|y|}{s^\frac{1}{2}}e^{-\frac{y^2}{4s}}\le C_1
$$
uniformly over ass $s > 0$ and $y\in \R$.
Then,
\begin{equation}\label{pi-2}
\begin{array}{lll}
\displaystyle
\int\limits_{\Pi_2}\!\!\!&
\displaystyle
\!\frac{|y|\,|z|}{s^\frac{3}{2}r^\frac{3}{2}}
e\big.^{-\frac{y^2}{4s}-\frac{z^2}{4r}}
\Phi\Big(\frac{y\!-\!z}{\eps^{\frac{1}{2}-\frac{\alpha}{4}}}, \frac{s\!-\!r}{\eps^{\frac{\alpha}{2}-1}}\Big)\,dy\,dz\,ds\,dr\\[3mm]
& \displaystyle \le C_1\int\limits^{T_\eps}_{\eps^{\varkappa_1}}\int\limits_0^{\eps^{\varkappa_1}}
\int\limits_{\mathbb R}
\,\frac{|z|\,dz\,dr\,ds}{s\,r^\frac{3}{2}}e^{-\frac{z^2}{r}}\int\limits_{\mathbb R}\Phi\Big(\frac{y\!-\!z}{\eps^{\frac{1}{2}-\frac{\alpha}{4}}}, \frac{s\!-\!r}{\eps^{\frac{\alpha}{2}-1}}\Big)\,dy\\[7mm]
& \displaystyle  =C_1\eps^\varkappa
\int\limits^{T_\eps}_{\eps^{\varkappa_1}}
\int\limits_0^{\eps^{\varkappa_1}}
\int\limits_{\mathbb R}
\,e^{-\frac{z^2}{r}}\,\overline\Phi\Big( \frac{s\!-\!r}{\eps^{\frac{\alpha}{2}-1}}\Big)\frac{|z|\,dz\,dr\,ds}{s\,r^\frac{3}{2}}
\\[7mm]
& \displaystyle
=CC_1\eps^\varkappa
\int\limits^{T_\eps}_{\eps^{\varkappa_1}}
\int\limits_0^{\eps^{\varkappa_1}}
\,\overline\Phi\Big( \frac{s\!-\!r}{\eps^{\frac{\alpha}{2}-1}}\Big)\frac{dr\,ds}{s\,r^\frac{1}{2}}
\leq
CC_1\eps^\varkappa
\int\limits^{T_\eps}_{\eps^{\varkappa_1}}
\int\limits_0^{\eps^{\varkappa_1}}
\,\widehat{\overline\Phi}\Big( \frac{s}{\eps^{\frac{\alpha}{2}-1}}\Big)\frac{dr\,ds}{s\,r^\frac{1}{2}}
\\[7mm]
& \displaystyle
=2CC_1\eps^\varkappa\eps^{\frac{\varkappa_1}{2}}
\int\limits^{T_\eps}_{\eps^{\varkappa_1}}
\,\widehat{\overline\Phi}\Big( \frac{s}{\eps^{\frac{\alpha}{2}-1}}\Big)\frac{ds}{s}
\le
2CC_1\eps^\varkappa\eps^{\frac{\varkappa_1}{2}}
\int\limits^{\infty}_{\eps^{\varkappa_1+2\varkappa}}
\,\widehat{\overline\Phi}(s)\frac{ds}{s}
\\[7mm]
& \displaystyle
\le C_2 (\varkappa_1+2\varkappa)\eps^{\varkappa+\frac{\varkappa_1}{2}} |\log\eps|;
\end{array}
\end{equation}
here $\overline\Phi(t)=\int_{\mathbb R}\Phi(x,t)dx$,
and $\widehat{\overline\Phi}(t)$ stands for
$\max\{\overline{\Phi}(s)\,:\,t-1\leq s\leq t\}$.
A similar estimate holds true for the integral over $\Pi_3$. Therefore,
\begin{equation}\label{y-est1}
\lim\limits_{\eps\to0}\ \sup\limits_{0<t\le+\infty} \Big|\int\limits_{\Pi_1\cup\Pi_2\cup\Pi_3}\frac{y}{s^{3/2}} \frac{z}{r^{3/2}} e\big.^{-\frac{y^2}{4s}-\frac{z^2}{4r}}
\Phi\Big(\frac{y\!-\!z}{\eps^{\frac{1}{2}-\frac{\alpha}{4}}}, \frac{s\!-\!r}{\eps^{\frac{\alpha}{2}-1}}\Big)
\,dy\,dz\,ds\,dr\Big|=0.
\end{equation}

We also have
\begin{equation}\label{y-est2}
\begin{array}{ll}
\displaystyle
\int\limits_{\Pi_1\cup\Pi_2}\frac{z^2}{s^{3/2} r^{3/2}} e\big.^{-\frac{z^2}{4s}-\frac{z^2}{4r}}
\Phi\Big(\frac{y\!-\!z}{\eps^{\frac{1}{2}-\frac{\alpha}{4}}}, \frac{s\!-\!r}{\eps^{\frac{\alpha}{2}-1}}\Big)
\,dy\,dz\,ds\,dr \\[7mm]
\displaystyle
=
C\eps^\varkappa\int\limits_{0}^{T_\eps}\int\limits_0^{\eps^{\varkappa_1}}
\overline\Phi\Big(\frac{s-r}{\eps^{-2\varkappa}}\Big)\frac{ds\,dr}{(s+r)^\frac{3}{2}}
=C\int\limits_{0}^{\eps^{2\varkappa}T_\eps} \int\limits_0^{\eps^{\varkappa_1+2\varkappa}}
\overline\Phi(s-r)\frac{ds\,dr}{(s+r)^\frac{3}{2}}
\\[7mm]
\displaystyle
\le C\int\limits_{0}^1 \int\limits_0^{\eps^{\varkappa_1+2\varkappa}}
\overline\Phi(s-r)\frac{ds\,dr}{(s+r)^\frac{3}{2}} +
C\int\limits_1^\infty \int\limits_0^{\eps^{\varkappa_1+2\varkappa}}
 \overline\Phi(s-r)\frac{ds\,dr}{(s+r)^\frac{3}{2}}
\le C_4\eps^\varkappa
\end{array}
\end{equation}
Combining this estimate with a similar estimate for the integral over $\Pi_1\cup\Pi_3$, we obtain
\begin{equation}\label{y-est3}
\lim\limits_{\eps\to0}\ \sup\limits_{0<t\le+\infty} \Big|\int\limits_{\Pi_1\cup\Pi_2\cup\Pi_3}\frac{z^2}{s^{3/2} r^{3/2}} e\big.^{-\frac{z^2}{4s}-\frac{z^2}{4r}}
\Phi\Big(\frac{y\!-\!z}{\eps^{\frac{1}{2}-\frac{\alpha}{4}}}, \frac{s\!-\!r}{\eps^{\frac{\alpha}{2}-1}}\Big)
\,dy\,dz\,ds\,dr\Big|=0.
\end{equation}
In order to justify (\ref{est_cruci}) it remains to show that
\begin{equation}\label{pi-4}
\lim\limits_{\eps\to0}\ \sup\limits_{0<t\le+\infty}
\int\limits_{\Pi_4}\frac{\big|{yz} e\big.^{-\frac{y^2}{4s}-\frac{z^2}{4r}}
-{z^2}e\big.^{-\frac{z^2}{4s}-\frac{z^2}{4r}}\big|}{s^{3/2} r^{3/2}}
\Phi\Big(\frac{y\!-\!z}{\eps^{\frac{1}{2}-\frac{\alpha}{4}}}, \frac{s\!-\!r}{\eps^{\frac{\alpha}{2}-1}}\Big)
\,dy\,dz\,ds\,dr\Big|=0
\end{equation}
We first estimate
\begin{equs}
J_\eps(t)&:=\int\limits_{\Pi_4}\frac{|yz|}{s^{3/2} r^{3/2}}
e\big.^{-\frac{z^2}{4r}}
\big|e\big.^{-\frac{y^2}{4s}}
-e\big.^{-\frac{z^2}{4s}}\big|
\Phi\Big(\frac{y\!-\!z}{\eps^{\varkappa}}, \frac{s\!-\!r}{\eps^{-2\varkappa}}\Big)
\,dy\,dz\,ds\,dr\\
&\le \frac{1}{4}\int\limits_{\Pi_4}\frac{|yz|\,|z^2-y^2|}{s^{5/2} r^{3/2}}
e\big.^{-\frac{z^2}{4r}}
\Big(e\big.^{-\frac{y^2}{4s}}
+e\big.^{-\frac{z^2}{4s}}\Big)
\Phi\Big(\frac{y\!-\!z}{\eps^{\varkappa}}, \frac{s\!-\!r}{\eps^{-2\varkappa}}\Big)
\,dy\,dz\,ds\,dr \label{y-est5}\\
&\lesssim \eps^\varkappa\int\limits_{\Pi_4} \Big(\frac{|y|^3+|y\!-\!z|^3}{s^{5/2} r^{3/2}}e\big.^{-\frac{y^2}{4s}}
\!+\!
\frac{|z|^3+|y\!-\!z|^3}{s^{5/2} r^{3/2}}e\big.^{-\frac{z^2}{4s}}
\Big)e\big.^{-\frac{z^2}{4r}}
\Phi_1\Big(\frac{y\!-\!z}{\eps^{\varkappa}}, \frac{s\!-\!r}{\eps^{-2\varkappa}}\Big)
\,dy\,dz\,ds\,dr
\end{equs}
with $\Phi_1(x,t)=|x|\Phi(x,t)$; here we have used the inequality $|e^a-e^b|\le |b-a|(e^a+e^b)$ and
the estimates $|yz||y+z|\leq C(|y|^3+|y-z|^3)$ and $|yz||y+z|\leq C(|z|^3+|y-z|^3)$ that follow from the Young inequality. Let us estimate the integral
$$
\eps^\varkappa\int\limits_{\Pi_4} \frac{|y|^3}{s^{5/2} r^{3/2}}e\big.^{-\frac{y^2}{4s}}
e\big.^{-\frac{z^2}{4r}}
\Phi_1\Big(\frac{y\!-\!z}{\eps^{\varkappa}}, \frac{s\!-\!r}{\eps^{-2\varkappa}}\Big)
\,dy\,dz\,ds\,dr
$$
$$
\le C_3\eps^\varkappa\int\limits_{\Pi_4} \frac{1}{s r^{3/2}}
e\big.^{-\frac{z^2}{4r}}
\Phi_1\Big(\frac{y\!-\!z}{\eps^{\varkappa}}, \frac{s\!-\!r}{\eps^{-2\varkappa}}\Big)
\,dy\,dz\,ds\,dr\le
C_4\eps^{2\varkappa}\int\limits_{\eps^{\varkappa_1}}^\infty
\int\limits_{\eps^{\varkappa_1}}^\infty
\frac{1}{s r}
\overline\Phi_1\Big(\frac{s\!-\!r}{\eps^{-2\varkappa}}\Big)
\,ds\,dr
$$
$$
=C_4\eps^{2\varkappa}\int\limits_{\eps^{\varkappa_1+2\varkappa}}^\infty
\int\limits_{\eps^{\varkappa_1+2\varkappa}}^\infty
\overline\Phi_1\big(s\!-\!r\big)\frac{ds\,dr}{s r}\,
\le C_5 \eps^{2\varkappa}(\log\eps)^2;
$$
here $C_3=\max(x^3e^{-x^2})$, and $\overline\Phi_1(t)$ stands for $\int_{\mathbb R}\Phi_1(x,t)dx$. Other terms on the right-hand side of (\ref{y-est5}) can be estimated in a similar way. Thus we obtain
\begin{equation}\label{y-est6}
\lim\limits_{\eps\to0}\ \sup\limits_{0<t\le\infty}\,J_\eps(t)=0.
\end{equation}
The inequality
$$
\int\limits_{\Pi_4}\frac{|yz-z^2|}{s^{3/2} r^{3/2}}
e\big.^{-\frac{z^2}{4r}}
e\big.^{-\frac{z^2}{4s}}
\Phi\Big(\frac{y\!-\!z}{\eps^{\varkappa}}, \frac{s\!-\!r}{\eps^{-2\varkappa}}\Big)\,dy\,dz\,ds\,dr\le C\eps^\varkappa(\log\eps)^2
$$
can be obtained in the same way with a number of
simplifications.
This yields (\ref{est_cruci}).

It remains to notice that
\begin{equs}
\int\limits_{\eps^{\varkappa_1}}^{T_\eps}\!\! \int\limits_{\eps^{\varkappa_1}}^{T_\eps}\!&\!
\int\limits_{\mathbb R}\!\!\!\int\limits_{\mathbb R} \frac{z^2}{s^{3/2} r^{3/2}} e\big.^{-\frac{z^2}{4s}-\frac{z^2}{4r}}
\Phi\Big(\frac{y\!-\!z}{\eps^{\varkappa}}, \frac{s\!-\!r}{\eps^{-2\varkappa}}\Big)
\,dy\,dz\,ds\,dr\\
&=
\eps^\varkappa\int\limits_{\eps^{\varkappa_1}}^{T_\eps}\!\! \int\limits_{\eps^{\varkappa_1}}^{T_\eps}\!\!
\int\limits_{\mathbb R} \frac{z^2}{s^{3/2} r^{3/2}} e\big.^{-\frac{z^2}{4s}-\frac{z^2}{4r}}
\overline\Phi\Big(\frac{s\!-\!r}{\eps^{-2\varkappa}}\Big)
\,dz\,ds\,dr\\
&= C_0
\eps^\varkappa\int\limits_{\eps^{\varkappa_1}}^{T_\eps}\!\! \int\limits_{\eps^{\varkappa_1}}^{T_\eps}\overline\Phi\Big(\frac{s\!-\!r}{\eps^{-2\varkappa}}\Big) \frac{ds\,dr}{(s+r)^{3/2}}
=C_0\int\limits_{\eps^{\varkappa_1+2\varkappa}}^{\eps^{-\alpha}t} \int\limits_{\eps^{\varkappa_1+2\varkappa}}^{\eps^{-\alpha}t}
\overline\Phi(s\!-\!r)\frac{ds\,dr}{(s+r)^{3/2}} \\
&=C_0\int\limits_{0}^{\infty} \int\limits_{0}^{\infty}
\overline\Phi(s\!-\!r)\frac{ds\,dr}{(s+r)^{3/2}}
+R_\eps(t)
\end{equs}
with $C_0=\int_{\mathbb R}z^2e^{-z^2/4}\,dz$, and
$$
\lim\limits_{\eps\to 0}\sup\limits_{1\le t\le+\infty}|R_\eps(t)|=0.
$$
Combining the last two relations with
(\ref{repr_main}) and (\ref{est_cruci}), we obtain the desired statement.
\end{proof}

\begin{lemma}\label{le:holder_estimY}
For any $T>0$, any even integer $k\ge 2$, any $0<\beta<1/k$, any $p>k$ and any $\kappa > 0$, there exists a  constant $C$ such that
for all $0\le t\le T$, $\eps>0$,
\begin{align*}
\left(\E\|Y^\eps(t)\|_{0,p_\kappa}^p\right)^{1/p}\le C\
\eps^{\frac{\alpha}{4}(1-\kappa)}\;&,\quad
\left(\E\|\d_xY^\eps(t)\|_{0,p_\kappa}^p\right)^{1/p}\le C\
\eps^{-\kappa},\;\\
\bigl(\E \|\d_x Y^\eps(t)\|_{\beta,p_\kappa}^p\bigr)^{1/p} &\le C \eps^{-\kappa}\;.
\end{align*}
\end{lemma}

\begin{proof} We establish the estimates of the norms of $\partial_xY^\eps(t)$ only. The norm of
$Y^\eps(t)$ is estimated similarly.
Let  $q>1$ and $p=qk$. For any $x<y$, we have the identity
\begin{equ}
|\partial_xY^\eps(t,y)-\partial_xY^\eps(t,x)|^k=k
\int_x^y(\partial_xY^\eps(t,z)-\partial_xY^\eps(t,x))^{k-1}\partial^2_xY^\eps(t,z)dz\;.
\end{equ}
Raising this to the power $q$ and taking expectations, we obtain
\begin{equs}
\E(|&\partial_xY^\eps(t,y)-\partial_xY^\eps(t,x)|^{p})\le k^q
\left|\int_x^y(\partial_xY^\eps(t,z)-\partial_xY^\eps(t,x))^{k-1}\partial^2_xY^\eps(t,z)dz\right|^q\\
&\lesssim (y-x)^{q-1}\int_x^y\E\left(\left| (\partial_xY^\eps(t,z)-\partial_xY^\eps(t,x))^{k-1}\partial^2_xY^\eps(t,z)\right|^q\right)dz\\
&\lesssim (y-x)^q\sqrt{\E\bigl(\left|\partial_xY^\eps(t,x)\right|^{2q(k-1)}\bigr)
\E\left(\left|\partial^2_xY^\eps(t,x)\right|^{2q}\right)}
\lesssim (y-x)^q \eps^{-q}, \label{e:boundDiffY}
\end{equs}
where we have used the stationarity (in $z$) of the processes
$\partial_xY^\eps(t,z)$
and $\partial^2_xY^\eps(t,z)$, as well as the estimates \eqref{e:momentsYderiv} and
\eqref{e:momentsYderiv2}  from Lemma
\ref{le:momentsY}.

As a consequence of Kolmogorov's Lemma, there exists a stationary sequence of
positive random variables $\{\xi_n\}_{n \in \Z}$ such that for every $n \in \Z$, the bound
\begin{equ}
\sup_{x \in [n,n+1]} |\d_x Y^\eps(t,x)| \le \xi_n\;,
\end{equ}
holds almost surely,
and such that $\bigl(\E \xi_n^p\bigr)^{1/p} \lesssim \eps^{-1/k}$
for every $p \ge 1$.
The bound on $\|\d_x Y^\eps(t)\|_{0,p_\kappa}$ then follows at once.

The bound on $\|\d_x Y^\eps(t)\|_{\beta,p_\kappa}$ follows in virtually the
same way, using the fact that \eref{e:boundDiffY} also yields the bound
\begin{equ}
\sup_{x,y \in [n-1,n+1]} {|\d_x Y^\eps(t,x) - \d_x Y^\eps(t,y)| \over |x-y|^\beta} \le \tilde \xi_n\;,
\end{equ}
for some stationary sequence of random variables $\tilde \xi_n$ which has
all of its moments bounded in the same way as the sequence $\{\xi_n\}$.
\end{proof}



We further obtain the following bound on the ``negative H\"older norm'' of
$\d_x Y^\eps$:
\begin{corollary}\label{cor:minus}
For any $T>0$, $k$ being any even integer, $p>k$ and $\kappa=1/k$, there exists a constant $C_{T,p,\kappa}$ such that
\begin{equ}
\left(\E\|\d_xY^\eps(t)\|_{-{1\over 4},p_\kappa}^p\right)^{1/p}\le C_{T,p,\kappa}\
\eps^{\alpha/16 -\kappa}\;,
\end{equ}
for all
$0\le t\le T$, $\eps>0$. \qed
\end{corollary}
\begin{proof}
We note that
\begin{equs}
\|\d_xY^\eps(t)\|_{-{1\over 4},p_\kappa}&=\sup_{|x-y|\le1}\frac{|Y^\eps(t,x)-Y^\eps(t,y)|}{p_\kappa(x)|x-y|^{3/4}},
\\
\|Y^\eps(t)\|_{0,p_\kappa}&=\sup_x\frac{|Y^\eps(t,x)|}{p_\kappa(x)}\;,\qquad
\|\d_xY^\eps(t)\|_{0,p_\kappa}=\sup_x\frac{|\d_xY^\eps(t,x)|}{p_\kappa(x)}.
\end{equs}
We have, for $|x-y|\le1$,
\begin{align*}
\frac{|Y^\eps(t,x)-Y^\eps(t,y)|}{p_\kappa(x)|x-y|^{3/4}}&=
\left(\frac{|Y^\eps(t,x)-Y^\eps(t,y)|}{p_\kappa(x)}\right)^{1/4}
\left(\frac{|Y^\eps(t,x)-Y^\eps(t,y)|}{p_\kappa(x)|x-y|}\right)^{3/4}\\
&\le \left(\frac{|Y^\eps(t,x)|}{p_\kappa(x)}+C_\kappa\frac{|Y^\eps(t,y)|}{p_\kappa(y)}\right)^{1/4}
\left(C_\kappa\sup_{x\le z\le y}\frac{|\d_xY^\eps(t,z)|}{p_\kappa(z)}\right)^{3/4}.
\end{align*}
It remains to take supremums and apply H\"older's inequality.
\end{proof}

We have similar results for $Z^\eps(x,t)$.
\begin{lemma}\label{le:momentsZ}
For each $p\ge1$, there exists a constant $C$ such that for all $\eps>0$, $t\ge0$, $x\in\R$,
$$
\left[\E\left(\left| Z^\eps(x,t)\right|^2\right)\right]^{1/2}\le C \bigl(1+t^{2}\bigr)\eps^{\gamma\alpha/2}.
$$
\end{lemma}

\begin{proof}
The main ingredient in the proof is a bound on the correlation function of
the right hand side of the equation for $Z^\eps$, which we denote by
\begin{equ}
\Lambda_\eps(z,z') = \Cov \bigl(|\d_x Y^\eps(z)|^2, |\d_x Y^\eps(z')|^2\bigr)\;.
\end{equ}
Inserting the definition of $Y^\eps$, we obtain the identity
\begin{equ}
\Lambda_\eps(z,z') = \int\cdots \int \tilde P(z-z_1)\tilde P(z-z_2)\tilde P(z'-z_3)\tilde P(z'-z_4) \Psi^{(4)}_\eps (z_1,\cdots,z_4)\,dz_1\cdots dz_4\;,
\end{equ}
where
\begin{equ}
\tilde P(z) = \tilde P(x,t) = \d_x p_t(x)\;,
\end{equ}
with $p_t$ the standard heat kernel and
\begin{equ}
\Psi^{(4)}_\eps (z_1,\cdots,z_4) = \eps^{-2-\alpha }\Psi^{(4)} \Bigl({x_1\over\eps},\cdots,{x_4\over\eps},{t_1\over\eps^\alpha},\cdots,{t_4\over\eps^\alpha}\Bigr)\;.
\end{equ}
Here, we used the shorthand notation $z_i = (x_i,t_i)$, and integrals over $z_i$ are
understood to be shorthand for $\int_0^t \int_\R\, dx_i\,dt_i$. We now make use of Lemma~\ref{ass:cor4}, which allows
to factor this integral as
\begin{equ}
|\Lambda_\eps(z,z')| \lesssim \Bigl(\eps^{-1-{\alpha\over 2} }\int \int \tilde P(z-z_1)\tilde P(z'-z_3) \rho_\eps(z_1-z_3)\,dz_1\, dz_3\Bigr)^2
\eqdef \tilde \rho_\eps^2(z,z')\;,
\end{equ}
where we used the shorthand notation
\begin{equ}
\rho_\eps(x,t) = \rho \Bigl({x\over \eps}, {t\over \eps^\alpha}\Bigr)\;.
\end{equ}
We will show below that the following bound holds:

\begin{lemma}\label{lem:boundreps}
\begin{equ}
\tilde \rho_\eps(z,z') \lesssim \Bigl(1 \wedge {\eps^{\alpha\gamma/2} \over d_p^\gamma(z,z')}\Bigr) + (1+t+t')\eps^{\alpha/2} \eqdef \zeta_
\eps(z-z') + (1+t+t')\eps^{\alpha/2}\;,
\end{equ}
where $d_p$ denotes the parabolic distance given by
\begin{equ}
d_p(z,z')^2 = |x-x'|^2 + |t-t'|\;.
\end{equ}
\end{lemma}

Taking this bound for granted, we write as in the proof of Lemma~\ref{le:momentsY} $Z^\eps = Z^\eps_- + \sum_{n > 0} Z^\eps_n$ with
\begin{equ}
Z^\eps_n(z) = 2^{-2n} \int \phi_n(z-z')\, \bigl(|\d_x Y^\eps(z')|^2 - \bar V_\eps(t')\bigr)\,dz'\;,
\end{equ}
and similarly for $Z^\eps_-$. Squaring this expression and inserting the bound from Lemma~\ref{lem:boundreps},
we obtain
\begin{equs}
\E |Z^\eps_n(z)|^2 &\lesssim 2^{-4n} \int\int \phi_n(z-z')\phi_n(z-z'')\, \bigl(\zeta_\eps^2(z'-z'') + (1+t'+t'')^2\eps^\alpha\bigr) \,dz'\,dz''\\
&\lesssim 2^{-n} \int \zeta_\eps^2(z')  \,dz' + 2^{-4n} (1+t)^4\eps^\alpha\;,
\end{equs}
where we made use of the scaling of $\phi_n$ given by \eref{e:defphin}. Performing the corresponding bound
for $Z^\eps_-$, we similarly obtain
\begin{equ}
\E |Z^\eps_-(z)|^2\lesssim t \int \zeta_\eps^2(z') \,dz' + (1+t)^4\eps^\alpha\;.
\end{equ}
The claim now follows from the bounds
\begin{equ}
\int \zeta_\eps^2(z') \,dz'  \le \int_0^t \int_\R {\eps^{\alpha\gamma} \over \bigl(|x|^2 + |s|\bigr)^\gamma} \,dx\,ds \lesssim \eps^{\alpha\gamma} t^{{3\over 2}-\gamma}\;.
\end{equ}
\end{proof}

\begin{proof}[of Lemma~\ref{lem:boundreps}]
Similarly to the proof of Lemma~\ref{le:momentsY}, we write
\begin{equ}
\tilde \rho_\eps(z,z') = \sum_{n_1\ge 0}\sum_{n_2\ge 0}\tilde \rho_\eps^{n_1,n_2}(z,z')\;,
\end{equ}
with
\begin{equ}
\tilde \rho_\eps^{n_1,n_2}(z,z') = \eps^{-1-{\alpha\over 2} }2^{-n_1-n_2}\int \int \tilde \phi_{n_1}(z-z_1)\tilde \phi_{n_2}(z'-z_2) \rho_\eps(z_1-z_2)\,dz_1\, dz_2\;.
\end{equ}
Here, for $n \ge 1$, $\tilde \phi_n$ is defined as in the proof of Lemma~\ref{le:momentsY}, whereas $\tilde \phi_0$ is different
from what it was there and is defined as
\begin{equ}
\tilde \phi_0(x,t) = \d_x p_t^-(x)\;.
\end{equ}
By symmetry, we can restrict ourselves to the case $n_1 \ge n_2$, which we will do in the sequel.
In the case where $n_2 > 0$, the above integral could be restricted to the set of pairs $(z_1, z_2)$, such that
their parabolic distance satisfies
\begin{equ}
d_p(z_1,z_2) \ge \bigl(d_p(z,z') - 2^{2-n_2}\bigr)_+\;,
\end{equ}
where $(\cdots)_+$ denotes the positive part of a number.

Replacing $\tilde \phi_{n_2}$ by its supremum and integrating out $\tilde \phi_{n_1}$ and $\rho_\eps$ yields the
bound
\begin{equ}
\tilde \rho_\eps^{n_1,n_2}(z,z') \lesssim \bigl(1+\delta_{n_2,0}(t+t')\bigr)2^{2n_2-n_1} \eps^{\alpha/2} \int_{A_\eps(n_2)} \rho(z_3)\,dz_3\;,
\end{equ}
where $A_\eps(0) = \R^2$ and
\begin{equ}
A_\eps(n_2) = \bigl\{z_3\,:\, d_p(0,z_3) \ge \eps^{-\alpha/2}\bigl(d_p(z,z') - 2^{2-n_2}\bigr)_+\bigr\}\;,
\end{equ}
for $n_2 > 0$. (Remark that the prefactor $1+t+t'$ is relevant only in the case $n_1=n_2 = 0$.)
It follows from the integrability of $\rho$ that one always has the bound
\begin{equ}[e:goodBound]
\tilde \rho_\eps^{n_1,n_2}(z,z') \lesssim \bigl(1+\delta_{n_2,0}(t+t')\bigr) 2^{2n_2-n_1} \eps^{\alpha/2}\;.
\end{equ}
Moreover, we deduce from Assumption~\ref{ass:cor} that, whenever $n_2 > 0$ and $d(z,z') \ge 2^{3-n_2}$, one has the improved bound: for any $\gamma>0$,
\begin{equ}[e:boundN2]
\tilde \rho_\eps^{n_1,n_2}(z,z') \lesssim 2^{2n_2-n_1} \eps^{\alpha/2} \Bigl(1 \wedge {\eps^{\alpha \gamma/2} \over d_p^\gamma(z,z')}\Bigr) \;.
\end{equ}
The bound \eref{e:goodBound} is sufficient for our needs in the case $n_2 = 0$, so we assume that $n_2 > 0$ from now on.

We now obtain a second bound on $\tilde \rho_\eps^{n_1,n_2}(z,z')$ which will be useful in the regime where $n_2$ is very large. Since the integral of $\tilde \phi_{n_1}$ is bounded independently of $n_1$, we obtain
\begin{equ}[e:generalBound]
\tilde \rho_\eps^{n_1,n_2}(z,z') \lesssim \eps^{-1-{\alpha\over 2} }2^{-n_1-n_2}\sup_{d_p(z_1,z) \le 2^{1-n_1}} \int \tilde \phi_{n_2}(z'-z_2) \rho_\eps(z_1-z_2)\,dz_2\;.
\end{equ}
We now distinguish between three cases, which depend on the size of $z-z'$.

\noindent{\bf Case 1: $d_p(z,z') \le \eps^{\alpha/2}$.} In this case, we proceed as in the proof of Lemma \ref{lem:moment_int_c},
which yields
\begin{equs}
\tilde \rho_\eps^{n_1,n_2}(z,z') &\lesssim \eps^{-1-{\alpha\over 2} }2^{-n_1-n_2}\sup_{z_1} \int \tilde \phi_{n_2}(z_2) \rho_\eps(z_2-z_1)\,dz_2\\
&\lesssim \eps^{-1-{\alpha\over 2} }2^{-n_1-n_2}\sup_{x_1} \int_{\R} \sup_{s} \rho_\eps(x_2-x_1,s) \int_{0}^t \tilde \phi_{n_2}(x_2,t_2) \,dt_2\,dx_2\\
&\lesssim \eps^{-1-{\alpha\over 2} }2^{-n_1}\int_{\R} \sup_{s} \rho_\eps(x_2,s)\,dx_2\lesssim \eps^{-{\alpha\over 2} }2^{-n_1}\;.
\label{e:simpleBound}
\end{equs}
\noindent{\bf Case 2: $|x-x'| \ge d_p(z,z')/2 \ge \eps^{\alpha/2}/2$.}
Note that in \eref{e:generalBound}, the argument of $\rho_\eps$ can only ever take values with
$|x_1 - x_2| \in B_\eps(n_2)$ where
\begin{equ}
B_\eps(n_2) = \bigl\{\bar x \,:\, |\bar x| \ge \bigl(|x-x'| - 2^{2-n_2}\bigr)\bigr\}\;.
\end{equ}
As a consequence, we obtain the bound
\begin{equ}
\tilde \rho_\eps^{n_1,n_2}(z,z') \lesssim \eps^{-1-{\alpha\over 2} }2^{-n_1-n_2}\sup_{\bar x \in B_\eps(n_2)} \sup_{s \in \R} \rho_\eps(\bar x,s)\;.
\end{equ}
The case of interest to us for this bound will be $2^{6-n_2} \le \eps^{\alpha/2}$, in which case we deduce from
this calculation and Assumption~\ref{ass:cor} that
\begin{equ}
\tilde \rho_\eps^{n_1,n_2}(z,z') \lesssim  \eps^{-1-{\alpha\over 2} }2^{-n_1-n_2}
\Bigl({\eps \over d_p(z,z')}\Bigr)^\gamma\;,
\end{equ}
where $\gamma$ is an arbitrarily large exponent. Choosing $\gamma$ large enough, we conclude that
one also has the bound
\begin{equ}[e:boundN31]
\tilde \rho_\eps^{n_1,n_2}(z,z') \lesssim \eps^{-{\alpha\over 2} }2^{-n_1}\Bigl(1\wedge{\eps^{\alpha/2} \over d_p(z,z')}\Bigr)^\gamma\;,
\end{equ}
which will be sufficient for our needs.

\noindent{\bf Case 3: $|t-t'| \ge d_p^2(z,z')/2 \ge \eps^{\alpha}/2$.} Similarly, we obtain
\begin{equ}
\tilde \rho_\eps^{n_1,n_2}(z,z') \lesssim \eps^{-{\alpha\over 2} }2^{-n_1}\int_{\R} \sup_{s\in B_\eps'(n_2)} \rho_\eps(x_2,s)\,dx_2\;,
\end{equ}
where
\begin{equ}
B_\eps'(n_2) = \bigl\{s \,:\, |s| \ge \eps^{-\alpha} \bigl(|t-t'| - 2^{8-2n_2}\bigr)\bigr\}\;.
\end{equ}
Restricting ourselves again to the case $2^{6-n_2} \le \eps^{\alpha/2}$, this yields as before
\begin{equ}[e:boundN32]
\tilde \rho_\eps^{n_1,n_2}(z,z') \lesssim  \eps^{-{\alpha\over 2} }2^{-n_1} \Bigl(1 \wedge  {\eps^{\alpha/2} \over d_p(z,z')}\Bigr)^\gamma\;.
\end{equ}
It now remains to sum over all values $n_1 \ge n_2\ge 0$.

For $n_2 = 0$, we sum the bound \eref{e:goodBound}, which yields
\begin{equ}
\sum_{n_1 \ge 0} \tilde \rho_\eps^{n_1,0}(z,z') \le (1+ t+t') \eps^{\alpha/2}\;.
\end{equ}
In order to sum the remaining terms, we first consider the case $d_p(z,z') < \eps^{\alpha/2}$.
In this case, we use \eref{e:goodBound} and \eref{e:simpleBound} to deduce that
\begin{equ}
\sum_{n_1 \ge n_2} \tilde \rho_\eps^{n_1,n_2}(z,z') \lesssim 2^{n_2} \eps^{\alpha/2}\wedge2^{-n_2} \eps^{-\alpha/2}\;,
\end{equ}
so that in this case $\tilde \rho_\eps(z,z') \lesssim 1+(1+ t+t') \eps^{\alpha/2}$.

It remains to consider the case $d_p(z,z') \ge \eps^{\alpha/2}$. For this, we break the sum over $n_2$ in
three pieces:
\begin{equs}
N_1 &= \{n_2 \ge 1\,:\, 2^{-n_2} \ge d(z,z')/8\}\;,\\
N_2 &= \{n_2 \ge 1\,:\, 2^{-6} \eps^{\alpha/2} \le 2^{-n_2} < d(z,z')/8\}\;,\\
N_3 &= \{n_2 \ge 1\,:\, 2^{-n_2} < 2^{-6} \eps^{\alpha/2}\}\;.
\end{equs}
For $n_2 \in N_1$, we only make use of the bound \eref{e:goodBound}. Summing first over $n_1 \ge n_2$ and then
over $n_2 \in N_1$, we obtain
\begin{equ}
\sum_{n_2 \in N_1}\sum_{n_1 \ge n_2}\tilde \rho_\eps^{n_1,n_2}(z,z') \lesssim {\eps^{\alpha/2} \over d_p(z,z')}\;.
\end{equ}
For $n_2 \in N_2$, we only make use of the bound \eref{e:boundN2}. Summing again first over $n_1 \ge n_2$ and then
over $n_2 \in N_1$, we obtain
\begin{equ}
\sum_{n_2 \in N_2}\sum_{n_1 \ge n_2}\tilde \rho_\eps^{n_1,n_2}(z,z') \lesssim {\eps^{\alpha\gamma /2} \over d_p^\gamma(z,z')}\;.
\end{equ}
In the last case, we similarly use either \eref{e:boundN31} or \eref{e:boundN32}, depending on whether
$|x-x'| \ge d_p(z,z')/2$ or $|t-t'| \ge d_p^2(z,z')/2$, which yields again
\begin{equ}
\sum_{n_2 \in N_3}\sum_{n_1 \ge n_2}\tilde \rho_\eps^{n_1,n_2}(z,z') \lesssim {\eps^{\alpha\gamma /2} \over d_p^\gamma(z,z')}\;.
\end{equ}
Combining the above bounds, the claim follows.
\end{proof}

\begin{lemma}\label{le:holder_estimZ}
For any $T>0$, $p\ge1$, $\kappa>0$, $0<\beta<1$, there exists a  constant $C_{T,p,\kappa,\beta}$ such that
for all $0\le t\le T$, $\eps>0$,
\begin{equ}
\bigl(\E \|\d_x Z^\eps(t)\|_{\beta,p_\kappa}^p\bigr)^{1/p} \le C_{T,p,\kappa,\beta} \eps^{-\kappa}\;.
\end{equ}
\end{lemma}

\begin{proof}
This is a corollary of Lemma~\ref{le:holder_estimY} and Proposition~\ref{prop:Heat}.
\end{proof}

As a Corollary we deduce
\begin{corollary}\label{co:Z}
For any $T>0$, $p\ge1$, $\kappa>0$, there exists a constant $C_{T,\kappa}$ such that for all
$0\le t\le T$, $\eps>0$,
\begin{equ}
\bigl|\E \|Z^\eps(t)\|_{0,p_\kappa}^p\bigr|^{1/p} \le C_{T,\kappa} \eps^{\alpha/2-\kappa}\;,\quad
\bigl|\E \|\d_x Z^\eps(t)\|_{0,p_\kappa}^p\bigr|^{1/p} \le C_{T,\kappa} \eps^{\alpha/4-\kappa}\;.
\end{equ}
\end{corollary}

We will need moreover
\begin{corollary}\label{co:Zto0}
As $\eps\to0$, $Z^\eps(x,t)\to0$ in probability, locally uniformly in $(x,t)$.
\end{corollary}
\begin{proof}
It follows from estimate \eqref{e:momentsYderiv} that for any $p>1$ and any bounded subset $K\subset \mathbb R\times\mathbb R^+$, there exists a constant $C_{p,K}$
such that
$$
\E\bigg(\int\limits_K \big(|\partial_xY^\eps(x,t)|^2-\overline V^\eps\big)^p dxdt\bigg)\le C_{p,K}.
$$
Then, by the Nash estimate, we obtain
\begin{equation}\label{e:narv1}
\E \|Z^\eps\|_{C^\gamma(K)}\le C_{K},
\end{equation}
where the H\"older exponent $\gamma>0$ and $C_{K}$ do not depend on $\eps$. As a consequence of the first estimate of Corollary \ref{co:Z} we have
\begin{equation}\label{e:narv2}
\E\|Z^\eps\|^p_{L^p(K)}\le C_{p,K}\eps^{p(\alpha/2-\kappa)}.
\end{equation}
Combining \eqref{e:narv1} and \eqref{e:narv2} one can easily derive the required convergence.
\end{proof}

\section{Proof of the main result}
\label{sec:final}




Before concluding with the proof of our main theorem, we prove a result for a parabolic heat equation with coefficients which live in spaces of weighted H\"older continuous functions.

We consider an abstract evolution equation of the type
\begin{equ}[e:abstract]
\d_t u = \d_x^2 u + F\,\d_x u + G\, u\;,
\end{equ}
where $F$ and $G$ are measurable functions of time,
taking values in $\CC^{-\beta}_{p_\kappa}$
for some suitable $\kappa > 0$ and $\beta < {1\over 2}$.
The main result of this section is the following:

\begin{theorem}\label{theo:linearFG}
Let $\beta$ and $\kappa$ be positive numbers such that $\beta + \kappa < {1\over 2}$
and
let $F$ and $G$ be functions in $L^p_\loc(\R_+, \CC^{-\beta}_{p_\kappa})$ for
every $p \ge 1$.

Let furthermore $\ell \in \R$ and $u_0 \in \CC^{3/2}_{e_\ell}$.
Then, there exists a unique global mild solution to \eref{e:abstract}. Furthermore, this solution is continuous with
values in $\CC^{3/2}_{e_m}$ for every $m < \ell$ and the map
$(u_0, F,G) \mapsto u$ is jointly continuous in these topologies.
\end{theorem}

\begin{proof}
We will show a slightly stronger statement, namely that for every $\delta > 0$ sufficiently small,
the mild solution has the property that $u_t \in \CC^{\gamma}_{e_{\ell-\delta t}}$ for
$t \in [0,T]$ for arbitrary values of $T>0$. We fix $T$, $\delta$ and $\ell$ from now on.

We then write
\begin{equ}
\$u\$_{\delta,\ell,T} \eqdef \sup_{t \in [0,T]} \|u_t\|_{{3\over 2},e_{\ell-\delta t}}\;,
\end{equ}
and we denote by $\CB_{\delta,\ell,T}$ the corresponding Banach space.
With this notation at hand, we define a map $\CM_T \colon \CB_{\delta,\ell,T} \to \CB_{\delta,\ell,T}$ by
\begin{equ}
\bigl(\CM_T u\bigr)_t = \int_0^t P_{t-s} \bigl(F_s \,\d_x u_s + G_s\, u_s\bigr)\,ds\;,\qquad t \in [0,T]\;.
\end{equ}
It follows from Proposition~\ref{prop:Heat} that we have the bound
\begin{equ}
\big\|\bigl(\CM_T u\bigr)_t \big\|_{{3\over2},e_{\ell-\delta t}} \le C\int_0^t (t-s)^{-{3 + 2\beta \over 4}}\bigl\|F_s \,\d_x u_s + G_s\, u_s\bigr\|_{-\beta,e_{\ell-\delta t}}\,ds\;.
\end{equ}
Combining Proposition~\ref{prop:multHol} with \eref{e:idexp} and \eref{e:boundpe}, we furthermore obtain the bound
\begin{equs}
\bigl\|F_s \,\d_x u_s\bigr\|_{-\beta,e_{\ell-\delta t}} &\le C \bigl(\delta |t-s|\bigr)^{-\kappa} \|F_s\|_{-\beta,p_\kappa} \bigl\|\d_x u_s\bigr\|_{{1\over2},e_{\ell-\delta s}} \\
& \le C \bigl(\delta |t-s|\bigr)^{-\kappa} \|F_s\|_{-\beta,p_\kappa} \$u\$_{\delta,\ell,T}\;,
\end{equs}
where the proportionality constant $C$ is uniformly bounded for $\delta \in (0,1]$ and bounded $\ell$ and $s$.
A similar bound holds for $G_s u_s$ so that, combining these bounds and using
H\"older's inequality for the integral over $t$, we obtain the existence of
constants $\zeta > 0$ and $p>1$ such that the bound
\begin{equ}
\$\CM_T u\$_{\delta,\ell,T} \le C \delta^{-\kappa} T^{\zeta} \bigl(\|F\|_{L^p(\CC^{-\beta}_{p_\kappa})} + \|G\|_{L^p(\CC^{-\beta}_{p_\kappa})}\bigr) \$u\$_{\delta,\ell,T}\;,
\end{equ}
holds.
Since the norm of this operator is strictly less than $1$ provided that $T$ is small enough, the short-time existence and uniqueness of solutions follow from Banach's fixed point theorem. The existence of solutions up to the final time $T$ follows
by iterating this argument, noting that the interval of short-time existence
restarting from $u(t)$ at time $t$ can be bounded from below by a constant
that is uniform over all $t \in [0,T]$, as a consequence of the linearity of the
equation.

Actually, we obtain the bound
\begin{equ}
\|u_t\|_{{3\over 2},e_{\ell-\delta t}} \lesssim \exp \bigl(C t \bigl(\|F\|_{L^p(\CC^{-\beta}_{p_\kappa})} + \|G\|_{L^p(\CC^{-\beta}_{p_\kappa})}\bigr)^{1/\zeta}\bigr)\|u_0\|_{{3\over2},e_\ell}\;,
\end{equ}
where the constants $C$ and $\zeta$ depend on the choice of $\ell$ and $\delta$.

The solutions are obviously linear in $u_0$ since the equation is linear in $u$.
It remains to show that the solutions also depend continuously on $F$ and $G$.
Let $\bar u$ be the solution to the equation
\begin{equ}[e:abstract2]
\d_t \bar u = \d_x^2 \bar u + \bar F\,\d_x \bar u + \bar G\, \bar u\;,
\end{equ}
and write $\rho =  u - \bar u$. The difference $\rho$ then satisfies the equation
\begin{equ}
\d_t \rho = \d_x^2 \rho + F\,\d_x \rho  + G\, \rho + (F - \bar F)\,\d_x \bar u
+ (G - \bar G)\,\bar u\;,
\end{equ}
with zero initial condition.
Similarly to before, we thus have
\begin{equ}
\rho_t = \bigl(\CM_T \rho\bigr)_t +  \int_0^t P_{t-s} \bigl((F_s - \bar F_s)\,\d_x \bar u_s
+ (G_s - \bar G_s)\,\bar u_s\bigr)\,ds\;.
\end{equ}
It follows from the above bounds that
\begin{equ}
\$\rho\$_{\delta,\ell,T} \lesssim  \$\CM_T \rho\$_{\delta,\ell,T}
+ C \delta^{-\kappa} T^\zeta \bigl(\|F-\bar F\|_{L^p(\CC^{-\beta}_{p_\kappa})} + \|G-\bar G\|_{L^p(\CC^{-\beta}_{p_\kappa})}\bigr) \$\bar u\$_{\delta,\ell,T}\;.
\end{equ}
Over short times, the required continuity statement thus follows at once. Over
fixed times, it follows as before by iterating the argument.
\end{proof}

\begin{remark}
In principle, one could obtain a similar result for less regular initial conditions,
but this does not seem worth the additional effort in this context.
\end{remark}

We now have finally all the ingredients in place to give the proof of our main result.

\begin{proof}[of Theorem~\ref{th:main}]
We apply Theorem~\ref{theo:linearFG} with $\beta = {1\over 4}$ and $\kappa = {1\over 8}$.
Note that the equation \eref{e:veps} for $v^\eps$ is precisely of the form
\eref{e:abstract} with
\begin{equ}
F = 2 \d_x Y^\eps + 2\d_x Z^\eps\;,\qquad G = |\d_x Z^\eps|^2 + 2\,\d_x Z^\eps \d_x Y^\eps\;.
\end{equ}
It follows from Corollaries~\ref{cor:minus} and \ref{co:Z} that, for every $p > 0$ and $\delta > 0$,
one has the bound
\begin{equ}
\Bigl|\E \int_0^T \|F\|_{\beta,p_\kappa}^p\,dt\Bigr|^{1/p} \lesssim \eps^{{\alpha \over 16}-\delta}\;,
\end{equ}
say. Similarly, it follows from Lemma~\ref{le:holder_estimY} and Corollary~\ref{co:Z} that one actually has the bound
\begin{equ}
\Bigl|\E \int_0^T \|G\|_{0,p_\kappa}^p\,dt\Bigr|^{1/p} \lesssim \eps^{{\alpha \over 4}-\delta}\;,
\end{equ}
which is stronger than what we required.
As a consequence of Theorem~\ref{theo:linearFG}, this
shows immediately that $v^\eps \to u$ in probability, locally uniformly both
in space and in time. We conclude by recalling that from Corollary \ref{co:Yto0}  and \ref{co:Zto0},   the
correctors $Y^\eps$ and $Z^\eps$ themselves converge locally uniformly to $0$
in probability.
\end{proof}

\bibliographystyle{./Martin}
\bibliography{./refs}

\end{document}